\documentclass[12pt]{article}

\usepackage{mathtools, amsthm, tikz, amssymb, latexsym, amsmath, amscd, amsfonts, array, lmodern, slantsc,enumerate,stmaryrd,rotating,caption,graphicx}
\usepackage{hyperref}
\usepackage{float}
\hypersetup{nesting=true,debug=true,naturalnames=true}
\usepackage[all]{xy}
\CompileMatrices

\usepackage{color}

\usepackage{tikz-cd}
\usepackage{relsize}
\usepackage{ulem} 
\usepackage{palatino, upgreek}
\usepackage{xcolor}

\usepackage[utf8]{inputenc}

\def\pn{[a_1,b_1]\cdots[a_{g-1},b_{g-1}]}
\def\qn{[b_{g-1},a_{g-1}]\cdots[b_1,a_1]}

\def\TC{\protect\operatorname{TC}}

\def\cat{\protect\operatorname{cat}}

\def\g{\protect\operatorname{\text{secat}}}

\newcommand{\ol}{\overline}
\newcommand{\z}{\mathbb{Z}}
\newcommand{\id}{\mathrm{id}}

\newcommand{\ot}{\otimes}

\DeclareMathOperator{\Hom}{Hom}

\newtheorem{proposition}{Proposition}[section]
\newtheorem{corollary}[proposition]{Corollary}
\newtheorem{definition}[proposition]{Definition}
\newtheorem{theorem}[proposition]{Theorem}
\newtheorem{remark}[proposition]{Remark}
\newtheorem{example}[proposition]{Example}
\newtheorem{lemma}[proposition]{Lemma}

\begin{document}

\title{Effective topological complexity of orientable-surface groups}
\author{Natalia Cadavid-Aguilar\footnote{The first author is grateful for support from FORDECYT grant 265667 ``Programa para un avance global e integral de la matem\'atica mexicana''.} \ and Jes\'us Gonz\'alez}

\date{\empty}

\maketitle

\begin{abstract}
We use rewriting systems to spell out cup-products in the (twisted) cohomology groups of a product of surface groups. This allows us to detect a non-trivial obstruction bounding from below the effective topological complexity of an orientable surface with respect to its antipodal involution. Our estimates are at most one unit from being optimal, and are closely related to the (regular) topological complexity of non-orientable surfaces.
\end{abstract}



\section{Introduction and main result}
A deep connection between the LS-category ($\cat$) of an aspherical space and the projective dimension of its fundamental group was established by Eilenberg and Ganea in~\cite{MR0085510}. In contrast, Farber's topological complexity TC (introduced in~\cite{Far,MR2074919}), a close relative of $\cat$ motivated by robotics, is currently missing a suitable analogue of the Eilenberg-Ganea $\cat$-result.

Historically, closed surfaces play a central TC-role. While the TC computation for the genus-$g$ orientable surface $\Sigma_g$ is an easy task, the genus-$g$ non-orientable surface $N_g$ sets a formidable challenge. This was started by Dranishnikov (\cite{MR3544546,MR3720882}), and completed by Cohen and Vandembroucq (\cite{MR3975552}) a few years ago in a {\it tour de force}. It is not difficult to see that $\TC(N_g)\in\{3,4\}$, and that the actual answer depends on whether a Berstein $\TC$-obstruction vanishes. This translates the topological problem into a purely algebraic one, for the obstruction is a theoretically well-understood element (a fourth cup-power) in a twisted cohomology group of $N_g\times N_g$, which can then be assessed in terms of group cohomology. The catch is that the relevant system of coefficients ---the fourth tensor power of the augmentation ideal of the integral group-ring of the fundamental group of $N_g$--- is (infinite and) highly complicated, so that a direct assessment of the obstruction becomes intractable. With great dexterity, Cohen and Vandembroucq manage to find a far smaller (actually finite rank) system of coefficients where the obstruction maps non-trivially, thus implying that $\TC(N_g)=4$ is the answer for $g\geq2$.

The above discussion suggests the possibility of finding a small (easily handable) system of coefficients on $N_g\times N_g$ that supports a nontrivial TC-obstruction implying $\TC(N_g)=4$. In fact, this paper aims at paving the algebraic grounds toward such a potential goal, in the stronger form presented below. 

Effective topological complexity, a variant of TC, was introduced in~\cite{MR3738182}. In general terms, when a group $G$ acts on a space $X$, B{\l}aszczyk-Kaluba's effective TC, $\TC^G(X)$, measures the topological instabilities of motion planning in $X$ when $G$-symmetries are taken into account. We are interested in $\TC^{\z_2}(\Sigma_g)$ due to its relation to $\TC(N_{g+1})$. Here $\z_2$ acts antipodally on $\Sigma_g$, so that $N_{g+1}$ is the corresponding orbit space. Namely,~\cite[Section~3, paragraph following Definition 3.1]{MR3738182} and~\cite[Theorem~3.10]{MR3312969} yield
\begin{equation}\label{lacompa}
\TC^{\z_2}(\Sigma_g)\leq\TC(N_{g+1})=4.
\end{equation}

The group-cohomology calculations in this paper allow us to detect a nontrivial obstruction in the cohomology of $\Sigma_g\times \Sigma_g$ (instead of $N_{g+1}\times N_{g+1}$), which refines~(\ref{lacompa}) to:

\begin{theorem}\label{etcso}
$3\leq\TC^{\z_2}(\Sigma_{g})\leq\TC(N_{g+1})=4$, for $g\geq 2$.
\end{theorem}

The low-genus case $N_1$ (i.e., the real projective plane) is special, for $\TC(N_1)=3$, as proved in~\cite{MR1988783}. In the effective-TC realm, the first two low-genus cases $\Sigma_0$ and $\Sigma_1$ are exceptional too: $\TC^{\z_2}(\Sigma_g)=g+1$ for $g=0,1$ (\cite{BGK,MR3738182}). Note that $\TC(N_{g+1})-\TC^{\z_2}(\Sigma_g)=2$ for $g=0,1$, while Theorem~\ref{etcso} yields $0\leq\TC(N_{g+1})-\TC^{\z_2}(\Sigma_g)\leq1$ for $g\geq2$.

Our proof of Theorem~\ref{etcso} uses this paper's cohomological calculations with simple $\mathbb{Z}$-coefficients. Remark~\ref{rf} at the end of the paper suggests that the full power of our twisted-coefficients cohomology results could lead to a proof of the equality $\TC^{\z_2}(\Sigma_g)=4$ for ``large'' genera. In view of~(\ref{lacompa}), this would give a simpler proof of Dranishnikov's and Cohen-Vandembroucq's corresponding assertion for $\TC(N_g)$.

As illustrated in~\cite{MR3386228}, our cup-product computations would also be relevant in other obstruction-theory flavored problems, such as in the study of non-principal torus bundles over products of surfaces.

\section{Preliminaries}
\subsection{The rewriting system}
In what follows $g$ stands for an integer greater than 1. Let $\Sigma_g$ (respectively, $\uppi_g$) stand for the closed orientable surface of genus $g$ (respectively, for the fundamental group of $\Sigma_g$). A standard presentation for $\uppi_g$ is
\begin{equation}\label{presentacion}
    \uppi_g = \langle\; a_1, b_1, \ldots, a_g, b_g\;\colon\,R_g\;\rangle,
\end{equation}
where $R_g=[a_1,b_1] \cdots [a_g,b_g]$, a product of commutators. We denote the neutral element in $\uppi_g$ by $1$, and the group inverse of an element $x\in\uppi_g$ by $\overline{x}$. In particular, any element in $\uppi_g$ can be expressed as a (possibly empty) word in the alphabet
\begin{equation}\label{alfabeto}
a_i,\overline{a_i},b_i,\overline{b_i},\;\; (1\le i\le g).
\end{equation}

There is a number of ways for producing, in an algorithmic fashion, normal forms for elements in $\uppi_g$. We use the method developed by Hermiller in terms of rewriting systems. The reader is referred to~\cite[Section~2]{MR1261121} and~\cite[Section~1]{MR1146597} for basic definitions about rewriting systems, and to~\cite[Section~3]{MR1261121} for a proof of the fact that rules~(\ref{rrsimples})--(\ref{R8}) below determine a finite complete rewriting system for $\uppi_g$.
\begin{align}
& a_i\overline{a_i}\to1, \;\;\overline{a_i}a_i\to1, \;\;b_i\overline{b_i}\to1, \;\;\overline{b_i}b_i\to1 \;\;\;(1\le i\le g); \label{rrsimples} \\ & a_gb_g\to\qn b_ga_g; \label{R5} \\
 & \overline{a_g}\overline{b_g}\to \overline{b_g}\overline{a_g}\qn; \label{R6} \\
 & a_g\overline{b_g}\to \overline{b_g}\pn a_g; \label{R7} \\
 & \overline{a_g}\qn b_g\to b_g\overline{a_g}. \label{R8}
\end{align}

Note that any word in the alphabet~(\ref{alfabeto}) containing no subword $\ell\overline{\ell}$ ($\ell$ in~(\ref{alfabeto})) and no letter $b_g$ nor $\overline{b_g}$ is necessarily in normal form. More generally, a word $\omega$ in the alphabet~(\ref{alfabeto}) is in normal form provided none of the words on the left of the rules~(\ref{rrsimples}) and~(\ref{R5})--(\ref{R8}) appears as a subword of~$\omega$. In such a case, any subword of~$\omega$ will automatically be in normal form.

\smallskip
We use the notation $N(x)$ to stand for the normal form of an element $x\in\uppi_g$. Thus $N(x)$ really stands for a word, rather than an actual element of $\uppi_g$. Yet, such a careful distinction will be overruled latter in the paper by using the same notation for either a word on the alphabet~(\ref{alfabeto}), or the element it represents in $\uppi_g$. The context will clearify the intended meaning.

\smallskip
For elements $x,y\in\uppi_g$, we will say that {\it $x$ ends like $y$} if, in terms of the concatenation product of words, $N(x)=N(z)N(y)$ for some $z\in\uppi_g$.

\subsection{Resolution for one factor}
Fox derivatives can be used to construct a minimal resolution for a group presented with a single relation. In the case of $\uppi_g$, the explicit resolution is spelled out next.

\begin{proposition}[{\cite[Proposition~2.1]{MR3386228}}]\label{spelledout}
A free resolution $M_*^g$ of the trivial $\uppi_g$-module $\z$ is
$$ 
    \begin{tikzpicture}[descr/.style={fill=white,inner sep=1.5pt}]
        \matrix (m) [
            matrix of math nodes,
            row sep=1em,
            column sep=2.5em,
            text height=1.5ex, text depth=0.25ex
        ]
        {0  &   M_2^g   &   M_1^g   &   M_0^g   &   \z  &   0,\\};
        \path[overlay,->, font=\scriptsize,>=latex]
        (m-1-1) edge (m-1-2)
        (m-1-2) edge[above] node {$d_2$}(m-1-3)
        (m-1-3) edge[above] node {$d_1$} (m-1-4)
        (m-1-4) edge[above] node {$\epsilon$} (m-1-5)
        (m-1-5) edge (m-1-6);
\end{tikzpicture}
$$ 
where each $\uppi_g$-module $M_i^g$ is free with basis indicated in the following table:

\medskip
\centerline{\begin{tabular}{|c|c|c|c|}\hline
\rule{0mm}{4mm}Module & $M_0^g$ & $M_1^g$ & $M_2^g$ \\[.2ex]
\hline
\rule{0mm}{4mm}basis & $\upchi$ & $\upalpha_i,\upbeta_i \;(1\le i\le g)$ & $\upomega$ \\ [.2ex]
\hline
\end{tabular}}

\bigskip\noindent Morphisms in $M_*^g$ are determined by $\,\epsilon(\upchi) = 1$, $d_1(\upalpha_{i} ) = (a_i-1)\upchi$, $ d_1(\upbeta_{i} ) = (b_i-1)\upchi$ and
\begin{equation}\label{ladedos}
d_2(\upomega)=\sum_{i=1}^g\left(\frac{\partial R_g}{\partial a_i} \upalpha_i+ \frac{\partial R_g}{\partial b_i} \upbeta_i\right),
\end{equation}
where partial derivative symbols stand for Fox derivatives in free differential calculus, and $R_g$ is the defining word relation in~(\ref{presentacion}).
\end{proposition}

Recall that the Fox derivative with respect to $u\in\{\upalpha_i,\upbeta_i\}$ of a word $v_1\cdots v_n$ in the alphabet~(\ref{alfabeto}) is defined recursively through the formula
$$\frac{\partial (v_1\cdots v_n)}{\partial u}=\frac{\partial v_1}{\partial u}+v_1\frac{\partial (v_2\cdots v_n)}{\partial u},$$
where the rule $\frac{\partial v}{\partial u}=0$ if $u\ne v$, as well as the rules $\frac{\partial \overline{v}}{\partial u}=-\overline{v}\frac{\partial v}{\partial u}$ and $\frac{\partial u}{\partial u}=1$  are in effect for $u,v\in\{\upalpha_i,\upbeta_i\}$.

Set $P_\ell:=c_1\cdots c_\ell$, with $c_\ell$ standing for the commutator $[a_\ell,b_\ell]$ for $1\le\ell\le g$.
It is an easy exercise to check that the ``total Fox derivative'' formula~(\ref{ladedos}) takes the explicit form
\begin{equation}\label{d2}
   d_2(\upomega)= 
    \mathlarger{\mathlarger{\sum}}_{i=1}^{g-1}\left[ \rule{0mm}{5mm}P_{i-1}\left(1-a_ib_i\ol{a_i}\right)\upalpha_{i}+P_{i-1}\left(a_i - c_{i}\right)\upbeta_{i} \right] + \left(P_{g-1}-b_g\right) \upalpha_{g}  + \left(P_{g-1}a_g-1\right) \upbeta_{g},
\end{equation}
where all words appearing in the last expression (after the obvious distribution of products in sums) are in normal form.

\section{Contracting homotopy via Hermiller's normal forms}
The goal of this section is to construct an explicit contracting homotopy $s_*=s_*^g$

\smallskip
\begin{equation}\label{hglksaytpowq}
    \begin{tikzpicture}[descr/.style={fill=white,inner sep=1.5pt}]
        \matrix (m) [
            matrix of math nodes,
            row sep=1em,
            column sep=2.5em,
            text height=1.5ex, text depth=0.25ex
        ]
        {0  &   M_2^g   &   M_1^g   &   M_0^g   &   \z  &   0,\\};
        \path[overlay,->, font=\scriptsize,>=latex]
        (m-1-1) edge (m-1-2)
        (m-1-2) edge[below] node {$d_2$}(m-1-3)
        (m-1-3) edge[below] node {$d_1$} (m-1-4)
        (m-1-4) edge[below] node {$\epsilon$} (m-1-5)
        (m-1-5) edge (m-1-6);
        \draw[overlay,->, font=\scriptsize,>=latex,bend right]
        (m-1-5) to node [above]{$s_{-1}$} (m-1-4);
        \draw[overlay,->, font=\scriptsize,>=latex,bend right]
        (m-1-4) to node [above]{$s_0$} (m-1-3);
        \draw[overlay,->, font=\scriptsize,>=latex,bend right]
        (m-1-3) to node [above]{$s_1$} (m-1-2);
\end{tikzpicture}
\end{equation}
i.e., morphisms $s_i\,(i=-1,0,1)$ of abelian groups satisfying the usual relations
\begin{equation}\label{contr4}
\epsilon s_{-1} = \id_{\z},\quad
d_1s_0 + s_{-1}\epsilon = \id_{M_0^g}, \quad
d_2s_1 + s_0d_1 = \id_{M_1^g} \quad\text{and\quad}
s_1d_2 = \id_{M_2^g}.
\end{equation}

\begin{remark}\label{aruonvcrujg}
The $s_{-1}$ and $s_0$ components of a contracting homotopy $s_*$ as above are fully described in~\cite{MR3386228}, but $s_1$ is described just enough to get a hold on the essential pieces of a diagonal approximation $M_*^g\to M_*^g\otimes M_*^g\hspace{.3mm}$ for $M_*^g$. Since we aim at producing diagonal approximations for each $M^{g_1,\ldots,g_n}_*=M^{g_1}_*\otimes\cdots\otimes M^{g_n}_*$, we are forced to describe $s_1$ in full.
\end{remark}

The morphism $s_{-1}:\z\to M_0^g$ is defined by $s_{-1}(1)=\upchi$, and the first condition in~(\ref{contr4}) is clearly satisfied. Slightly more elaborate is to define the morphism $s_0\colon M_0^g\to M_1^g$. We set
\begin{equation}\label{deritotalFox}
s_0(\upchi)= 0,  \;\;\;  s_0(a_i \upchi) = \upalpha_i,  \;\;\;  s_0(b_i \upchi) = \upbeta_i,  \;\;\;  s_0(\ol{a_i} \upchi) = -\ol{a_i} \upalpha_i,  \;\;\;  s_0(\ol{b_i} \upchi) = - \ol{b_i} \upbeta_i, 
\end{equation}
where $i\in\{1,\ldots,g\}$, and for $y\in\uppi_g$ with $N(y)=\ell_1 \cdots \ell_k$, we set
\begin{equation}\label{dtf}
s_0(y\upchi)=s_0(\ell_1 \ell_2 \cdots \ell_k\upchi) = s_0(\ell_1\upchi) + \ell_1 s_0(\ell_2\upchi) + \cdots + \ell_1 \cdots \ell_{k-1} s_0(\ell_k\upchi).    
\end{equation}
Definition in~(\ref{deritotalFox}) and~(\ref{dtf}) of the map $s_0$ as a ``total Fox derivative'' (in terms of normal forms) agrees with definitions in other sources (e.g.~\cite{MR1092227,MR920522}). 

\begin{remark}\label{formulaubarvsd2}
From~(\ref{dtf}) we see that, if a (concatenation) product of words $\omega_1\omega_2\cdots\omega_k$ is in normal form, then
\begin{equation}\label{dtfgeneral}
s_0(\omega_1\omega_2\cdots\omega_k\upchi)=s_0(\omega_1\upchi)+\omega_1s_0(\omega_2\upchi)+\cdots+\omega_1\omega_2\cdots\omega_{k-1}s_0(\omega_k\upchi).
\end{equation}
In fact, since~(\ref{deritotalFox}) implies the relation $s_0(\ell)+\ell s_0(\ol{\ell})=0$ for $\ell\in\{a_i,b_i\colon 1\le i\le g\}$, we have that~(\ref{dtfgeneral}) holds for any word $\omega_1\cdots\omega_k$, even if it is not in normal form, but as long as it can be put in normal form by applying type-(\ref{rrsimples}) rewriting rules.
\end{remark}

\begin{proposition}
The second condition in~(\ref{contr4}) holds true.
\end{proposition}
\begin{proof}
The second condition in~(\ref{contr4}) holds at $y\upchi$ whenever $y\in\{1, a_i, \ol{a_i}, b_i, \ol{b_i}\,|\,i=1,\ldots, g\}$, by definition. For $N(y)=\ell_1\cdots\ell_k$ with $k\geq2$, we have 
\begin{align*}
d_1s_0\left(\ell_1\cdots\ell_k\upchi\right)
&=d_1\left(s_0(\ell_1\upchi) + \ell_1 s_0(\ell_2\upchi) + \cdots + \ell_1 \cdots \ell_{k-1} s_0(\ell_k\upchi)\rule{0mm}{4mm}\right) \\
&=d_1s_0(\ell_1\upchi) + \ell_1 d_1s_0(\ell_2\upchi) + \cdots + \ell_1 \cdots \ell_{k-1} d_1s_0(\ell_k\upchi)\\
&=\left(\ell_1 - 1\right)\upchi+ \ell_1\left(\ell_2 - 1\right)\upchi+\cdots+
\ell_1 \cdots \ell_{k-1} \left(\ell_k - 1\right)\upchi\\
&=\ell_1 \cdots \ell_{k}\upchi-\upchi\,=\,y\upchi-s_{-1}\epsilon\left(y\upchi\right),
\end{align*}
which completes the proof.
\end{proof}

Before defining the morphism $s_1\colon M_1^g\to M_2^g$, we introduce a few auxiliary elements, and record a number of helpful relations between them. Set $U:=\ol{a_g}\ol{P_{g-1}}$ and, for $n\geq0$, $T_n:=a_gU^{-n}=a_g\left(P_{g-1}a_g\right)^n=\left(a_gP_{g-1}\right)^na_g$. Note that the right-most expressions in the definitions of $U$ and $T_n$ are words in normal form. Straightforward calculation using~(\ref{dtfgeneral}) (and~(\ref{d2}), in the case of~(\ref{hhh})) gives the relations
\begin{align}
s_0(\ol{a_g}^{m}\upchi)&=-(\ol{a_g}+\ol{a_g}^2+\cdots+\ol{a_g}^{m})\upalpha_g, \nonumber \\
s_0(b_g\ol{a_g}^m\upchi)&=\upbeta_g-b_g(\ol{a_g}+\ol{a_g}^2+\cdots+\ol{a_g}^m)\upalpha_g, \label{tqoiusazx} \\
s_0(U^{m+1}\upchi)&=(1+U+\cdots+U^m)s_0(U\upchi), \label{iuytghjkmnb}\\
s_0(\ol{U}\upchi)&=d_2(\upomega)+b_g\upalpha_g+(1-\ol{U})\upbeta_g \label{hhh} \\
&= \mathlarger{\mathlarger{\sum}}_{i=1}^{g-1}\left[ P_{i-1}\hspace{-1mm}\left(\rule{0mm}{4mm}\hspace{-1mm}\left(1-a_ib_i\ol{a_i}\right)\upalpha_{i} + \left(a_i - c_{i}\right)\upbeta_{i}\right)\right] + \ol{U}\ol{a_g}\upalpha_g,\nonumber
\end{align}
where the term $s_0(U\upchi)$ in~(\ref{iuytghjkmnb}) is given by
\begin{equation}
a_gs_0(U\upchi) = \mathlarger{\mathlarger{\sum}}_{i=1}^{g-1}\left[ \left(\prod_{k=i+1}^{g-1}c_{k}\right)^{\!\!\!-1}\!\!\left(\rule{0mm}{5mm}\!\!\left(b_i-\ol{c_{i}}\right)\upalpha_{i} + \left(1 - b_ia_i\ol{b_i}\right)\upbeta_{i}\right)\right] -\upalpha_g. \nonumber
\end{equation}
The apparent relationship between the last two expressions is formalized by the second observation in Remark~\ref{formulaubarvsd2}, which yields $0=s_0(\ol{U}U\upchi)=s_0(\ol{U}\upchi)+\ol{U}s_0(U\upchi)$, so
\begin{equation}\label{nataobs}
s_0(\ol{U}\upchi)=-\ol{U}s_0(U\upchi).
\end{equation}
Likewise,
\begin{equation}\label{laosjjsssjjss}
s_0(a_gU\upchi) = \upalpha_g + a_gs_0(U\upchi). 
\end{equation}
Further, the rewriting rules (\ref{R5})--(\ref{R8}) yield respectively 
\begin{equation}\label{aux1}
a_gb_g=a_gUb_ga_g,\;\; N(a_g\ol{b_g})=\ol{b_g}\,\ol{U}, \;\; N(\ol{a_g}\ol{b_g})=\ol{b_g}U, \;\;  N(Ub_g)=b_g\ol{a_g}.
\end{equation}
The word on the right-hand side of the first equality in~(\ref{aux1}) is one rewriting rule away from being in normal form. Yet, the first equality in~(\ref{aux1}) can be written as $N(\ol{U}b_g)=b_ga_g$, which has the flavor of the other three equations in~(\ref{aux1}). By iteration we get, for $m\ge0$,
\begin{equation}\label{lemma32}
N(\ol{U}^mb_g)=b_ga_g^m,\;\; N(a_g^m\ol{b_g})=\ol{b_g}\,\ol{U}^m, \;\; N(\ol{a_g}^m\ol{b_g})=\ol{b_g}U^m, \;\;  N(U^mb_g)=b_g\ol{a_g}^m.
\end{equation}

The morphism $s_1\colon M_1^g\to M_2^g$ is defined on a $\mathbb{Z}$-basis element $y\uplambda$ (here $y\in\uppi_g$ and $\uplambda \in \{\upalpha_{1} , \upbeta_{1} , \ldots, \upalpha_{g} , \upbeta_{g} \}$) by setting $s_1(y \uplambda)=0$, if $\uplambda\neq \upbeta_g$, and
$$
s_1(y \upbeta_{g} ) = \begin{cases}
\left(y \displaystyle\sum_{i=1}^{n+1}U^i\right)\upomega, 
& \mbox{if $y$ ends like $T_n$ but not like $T_{n+1}$, $n \geq 0$; } \\
-\left(y \displaystyle\sum_{i=0}^{n-1}\ol{U^i}\right)\upomega,
& \mbox{if $y$ ends like $U^{n}$ but not like $U^{n+1}$, $n \geq 1$; } \\
0,& \mbox{otherwise.}
\end{cases}
$$

\begin{proposition}\label{nodo1}
The third condition in~(\ref{contr4}) holds true at any $\mathbb{Z}$-basis element $y\uplambda\in M^g_1$ with $y\in\uppi_g$ and $\uplambda\in\{\upalpha_1,\upalpha_2,\ldots,\upalpha_g,\upbeta_1,\upbeta_2,\ldots,\upbeta_{g-1}\}$.
\end{proposition}
\begin{proof}
We only consider the case $\uplambda=\upalpha_i$ with $1\le i\le g$; the case $\uplambda=\upbeta_i$ with $1\le i<g$ is formally identical. Set $N(y)=\ell_1\cdots\ell_k$.

\medskip\noindent {\bf Case $\ell_k \neq \ol{a_i}\hspace{.35mm}$}: Note that $\ell_1 \cdots \ell_k a_i$ is in normal form. Since $s_1(y \upalpha_i) = 0$, we get
\begin{equation}\label{auxene}
s_0d_1 (y\upalpha_i) = s_0 (\ell_1 \cdots \ell_k a_i\upchi)- s_0 (\ell_1 \cdots \ell_k\upchi) = \ell_1 \cdots \ell_k \upalpha_i = y\upalpha_i-d_2s_1\left(y\upalpha_i\right).
\end{equation}

\medskip\noindent {\bf Case $\ell_k = \ol{a_i}\hspace{.35mm}$}: Note that $\ell_1 \cdots \ell_k a_i = \ell_1 \cdots \ell_{k-1}$, where the latter word is in normal form, so the calculation in~(\ref{auxene}) takes the form
\begin{align*}
s_0d_1(y \upalpha_i)
&= s_0(\ell_1 \cdots \ell_{k-1}\upchi) - s_0(\ell_1 \cdots \ell_{k}\upchi) = - \ell_1 \cdots \ell_{k-1} s_0(\ol{a_i}\upchi)\\
&= - \ell_1 \cdots \ell_{k-1} (- \ol{a_i} \upalpha_i) = y \upalpha_i= y \upalpha_i - d_2s_1(y\upalpha_i),
\end{align*}
again since $s_1(y \upalpha_i) = 0$.
\end{proof}

More involved is the proof of the analogue of Proposition~\ref{nodo1} for $\uplambda=\upbeta_g$:
\begin{proposition}
The third condition in~(\ref{contr4}) holds at any $\mathbb{Z}$-basis element $y\upbeta_g$ with $y\in\uppi_g$.
\end{proposition}
\begin{proof}
As above, set
\begin{equation}\label{provi}
N(y)=\ell_1\cdots\ell_k.
\end{equation}
The argument depends on the normal form of the coefficient of $\upchi$ in the first summand on the right-hand side of
\begin{equation}\label{kqqjsyyios}
s_0d_1(y\upbeta_g)=s_0(\ell_1\cdots\ell_kb_g\upchi)-s_0(\ell_1\cdots\ell_k\upchi).
\end{equation}

From the rewriting rules~(\ref{rrsimples})--(\ref{R8}), we see that the word $\ell_1\cdots\ell_kb_g$ is in normal form if and only if $\ell_k\not\in\{\ol{b_g},a_g\}$ and $y$ does not end like $U$. Moreover, in such a case, $s_1(y\upbeta_g)=0$ by definition, so the required conclusion follows just as in the first case in the proof of Proposition~\ref{nodo1}. Thus, it suffices to consider the following three mutually non-overlapping cases: {\bf \;(A)} $\ell_k=\ol{b_g}$, {\bf \;(B)} $y$ ends like $U$, and {\bf \;(C)} $\ell_k=a_g$.

\smallskip\noindent {\bf Case (A)\hspace{.35mm}}: This is the easiest instance, for $N(\ell_1\cdots\ell_k b_g)=\ell_1\cdots\ell_{k-1}$ is forced, and the required conclusion follows just as in the second case in the proof of Proposition~\ref{nodo1}. 

\smallskip\noindent {\bf Case (B)\hspace{.35mm}}: Assume that~(\ref{provi}) specializes to $N(y)=\ell_1\cdots\ell_rU^{n+1}$ for some $n\geq0$, with the (normal-form) word $\ell_1\cdots\ell_r$ not ending like $U$. Using~(\ref{lemma32}), the expression in~(\ref{kqqjsyyios}) becomes
\begin{align}
s_0d_1(y\upbeta_g)&=s_0(\ell_1\cdots\ell_r U^{n+1}b_g\upchi)-s_0(\ell_1\cdots\ell_rU^{n+1}\upchi)\nonumber\\&
=s_0(\ell_1\cdots\ell_rb_g\ol{a_g}^{n+1}\upchi)-s_0(\ell_1\cdots\ell_rU^{n+1}\upchi),\label{tbced}
\end{align}
whose analysis depends, in principle, on the normal form of
\begin{equation}\label{underb1}
\underbrace{\ell_1\cdots\ell_r}\underbrace{b_g\ol{a_g}^{n+1}}.
\end{equation}
Since the two underbraced words in~(\ref{underb1}) are in normal form, since $\ell_r\neq a_g$ (for $\ell_1\cdots\ell_rU^{n+1}$ is in normal form), and since $\ell_1\cdots\ell_r$ does not end like $U$, we see that~(\ref{underb1}) fails to be in normal form if and only if $\ell_r=\ol{b_g}$, in which case the normal form of~(\ref{underb1}) would be $\ell_1\cdots\ell_{r-1}\ol{a_g}^{n+1}$ (for $\ell_{r-1}\neq a_g$, as $\ell_1\cdots\ell_r$ is in normal form). Either way, Remark~\ref{formulaubarvsd2} yields
$$
s_0(\ell_1\cdots\ell_rb_g\ol{a_g}^{n+1}\upchi)=s_0(\ell_1\cdots\ell_r\upchi)+\ell_1\cdots\ell_rs_0(b_g\ol{a_g}^{n+1}\upchi),
$$
so that~(\ref{tbced}) becomes
\begin{align}
s_0d_1(y\upbeta_g)&=\ell_1\cdots\ell_r\left( s_0(b_g\ol{a_g}^{n+1}\upchi)-s_0(U^{n+1}\upchi)\right)\nonumber\\&
=\ell_1\cdots\ell_r\left(\upbeta_g-b_g\sum_{i=1}^{n+1}\ol{a_g}^i\upalpha_g-\sum_{i=0}^nU^is_0(U\upchi) \right),
\label{s0d1ybgB2}
\end{align}
where the latter equality uses~(\ref{tqoiusazx}) and~(\ref{iuytghjkmnb}). On the other hand, by definition, (\ref{hhh}) and~(\ref{nataobs}),
\begin{align}
d_2s_1(y\upbeta_g)
&=-y\sum_{i=0}^n\ol{U}^id_2(\upomega)=-\ell_1\cdots\ell_r\sum_{i=1}^{n+1}U^id_2(\upomega)\nonumber \\
&=-\ell_1\cdots\ell_r\sum_{i=1}^{n+1}U^i\left(s_0(\ol{U}\upchi)-b_g\upalpha_g-(1-\ol{U})\upbeta_g\right)\nonumber\\
&=-\ell_1\cdots\ell_r\sum_{i=1}^{n+1}U^i\left(-\ol{U}s_0(U\upchi)-b_g\upalpha_g-(1-\ol{U})\upbeta_g\right)\nonumber\\
&=\ell_1\cdots\ell_r\left(\sum_{i=0}^{n}U^is_0(U\upchi)+\sum_{i=1}^{n+1}U^ib_g\upalpha_g+(U^{n+1}-1)\upbeta_g   \right)\nonumber\\
&=\ell_1\cdots\ell_r\left(\sum_{i=0}^{n}U^is_0(U\upchi)+\sum_{i=1}^{n+1}b_g\ol{a_g}^i\upalpha_g-\upbeta_g   \right)+y\upbeta_g,\label{d2s1B1special}
\end{align}
where the latter equality uses~(\ref{lemma32}). The required conclusion is apparent from~(\ref{s0d1ybgB2}) and~(\ref{d2s1B1special}).

\medskip\noindent {\bf Case (C)\hspace{.35mm}}: We can assume that~(\ref{provi}) specializes to $N(y)=\ell_1\cdots\ell_rT_n=\ell_1\cdots\ell_ra_g\ol{U}^n$, for some $n\geq0$, with 
\begin{equation}\label{iqiqiqkakamxmx}
\mbox{the (normal-form) word $\ell_1\cdots\ell_r$ not ending like $a_gc_{1}\cdots c_{g-1}$.}
\end{equation}
Using~(\ref{R5}) and~(\ref{lemma32}), the expression in~(\ref{kqqjsyyios}) becomes
\begin{align}
s_0d_1(y\upbeta_g)&=s_0\left(\ell_1\cdots\ell_ra_g\ol{U}^nb_g\upchi\right)-s_0\left(\ell_1\cdots\ell_ra_g\ol{U}^n\upchi\right) \nonumber\\ &=s_0\left(\ell_1\cdots\ell_ra_gb_ga_g^n\upchi\right)-s_0\left(\ell_1\cdots\ell_ra_g\ol{U}^n\upchi\right) \nonumber\\ &=s_0\left(\ell_1\cdots\ell_r \ol{P_{g-1}} b_ga_g^{n+1}\upchi\right)-s_0\left(\ell_1\cdots\ell_ra_g\ol{U}^n\upchi\right),\label{tbced2}
\end{align}
whose analysis depends, in principle, on the normal form of
\begin{equation}\label{underb1bis}
\underbrace{\ell_1\cdots\ell_r}\underbrace{\ol{P_{g-1}}b_ga_g^{n+1}}.
\end{equation}
Since the two underbraced words in~(\ref{underb1bis}) are in normal form, and since $\ell_r\neq\ol{a_g}$ (for $\ell_1\cdots\ell_ra_g\ol{U}^n$ is in normal form), we see that the only option for~(\ref{underb1bis}) not to be in normal form is that rewriting rules~(\ref{rrsimples}) can iteratively be applied between the right factors of the first underbraced word and the left factors of the second underbraced word in~(\ref{underb1bis}), until reaching the normal form. Indeed,~(\ref{iqiqiqkakamxmx}) and the fact that $\ell_1\cdots\ell_r$ is in normal form assure that, if the iterative rewriting process with rules~(\ref{rrsimples}) reaches a stage where the portion $\ol{P_{g-1}}$ in the second underbraced word in~(\ref{underb1bis}) is cancelled out, then this process would either stop or, else, continue with the application of one final rewriting rule in~(\ref{rrsimples}), namely the rule $\ol{b_g}b_g\to1$ ---but not the rule~(\ref{R5}). Either way, the process stops producing the required normal form. Explicitly, and for completeness, we remark that all the possibilities  for this process to stop producing the normal form of~(\ref{underb1bis}) are
\begin{itemize}
\item[{\bf (C.1)}] $\ell_1 \cdots \ell_r\ol{P_{g-1}} b_ga_g^{n+1}$ is already in normal form, or else
\end{itemize}
there are integers $t$ and $s$, with $1 \leq t \leq r$ and $g-1 \geq s \geq 1$, such that one (and only one) of the following situations holds:
\begin{itemize}
\item[{\bf (C.2)}] $\ell_t \cdots \ell_r = \ol{b_s}c_{s+1} \cdots c_{g-2}c_{g-1}$ with $\ell_{t-1}\neq\ol{a_s}.$ (Note that the latter inequality holds vacuously if $t=1$. A similar observation applies in the next five cases.)
\item[{\bf (C.3)}] $\ell_t \cdots \ell_r = \ol{a_s}\ol{b_s}c_{s+1} \cdots c_{g-2}c_{g-1}$ with $\ell_{t-1}\neq b_s$.
\item[{\bf (C.4)}] $\ell_t \cdots \ell_r = b_s\ol{a_s}\ol{b_s}c_{s+1} \cdots c_{g-2}c_{g-1}$ with $\ell_{t-1}\neq a_s$.
\item[{\bf (C.5)}] $\ell_t \cdots \ell_r = c_{s}c_{s+1} \cdots c_{g-2}c_{g-1}$ with $s\geq2$ and $\ell_{t-1}\neq \ol{b_{s-1}}$.
\item[{\bf (C.6)}] $\ell_t \cdots \ell_r = c_{s}c_{s+1} \cdots c_{g-2}c_{g-1}$ with $s=1$, $\ell_{t-1}\neq a_g$ (in view of~(\ref{iqiqiqkakamxmx})), $\ell_{t-1}\neq\ol{a_1}$ (since $\ell_1\cdots\ell_r$ is in normal form) and $\ell_{t-1}\neq\ol{b_g}$.
\item[{\bf (C.7)}] $\ell_t \cdots \ell_r = \ol{b_g}P_{g-1}$, and necessarily $\ell_{t-1}\neq \ol{a_{g}}$ (as $\ell_1\cdots\ell_r$ is in normal form).
\end{itemize}
So we need to check the validness of the third condition in~(\ref{contr4}) at $y\upbeta_g$ in each of these seven possibilities. As in case (B), such a task is simplified by Remark~\ref{formulaubarvsd2} and the discussion above. Namely, in all seven cases, the $s_0$-value of~(\ref{underb1bis}) satisfies
\begin{equation*}
s_0(\ell_1 \cdots \ell_r\ol{P_{g-1}} b_ga_g^{n+1}\upchi)=s_0(\ell_1 \cdots \ell_r\upchi)+\ell_1\cdots\ell_rs_0(\ol{P_{g-1}}b_ga_g^{n+1}\upchi),
\end{equation*}
so that~(\ref{tbced2}) becomes
%
\begin{align*}
s_0d_1(y\upbeta_g)&=\ell_1\cdots\ell_r\left( s_0\left(\ol{P_{g-1}}b_ga_g^{n+1}\upchi\right)-s_0\left(a_g\ol{U}^n\upchi\right)\right)\\
&=\ell_1\cdots\ell_r\left( s_0\left(\ol{P_{g-1}}b_ga_g^{n+1}\upchi\right)-\upalpha_g-a_g\sum^n_{i=1}\ol{U}^{i-1}s_0(\ol{U}\upchi)\right),
\end{align*}
whereas~(\ref{hhh}) gives
\begin{align*}
d_2s_1(y\upbeta_g)&=\ell_1\cdots\ell_ra_g\sum_{i=0}^n\ol{U}^{i-1}\left( s_0(\ol{U}\upchi)-b_g\upalpha_g-(1-\ol{U})\upbeta_g \right)\\&=\ell_1\cdots\ell_ra_g\left(\sum_{i=0}^n\ol{U}^{i-1}\left( s_0(\ol{U}\upchi)-b_g\upalpha_g\right)+(\ol{U}^n-U)\upbeta_g \right).
\end{align*}
Thus $s_0d_1(y\upbeta_g)+d_2s_1(y\upbeta_g)=\ell_1\cdots\ell_r (A+a_g\ol{U}^n\upbeta_g)=\ell_1\cdots\ell_rA+y\upbeta_g=y\upbeta_g$, as
\begin{align*}
A&=s_0\!\left(\ol{P_{g-1}}b_ga_g^{n+1}\upchi\right)-\upalpha_g+a_gUs_0(\ol{U}\upchi)-a_g\sum_{i=0}^n\ol{U}^{i-1}b_g\upalpha_g-a_g U\upbeta_g
\end{align*}
vanishes. Indeed, use~(\ref{dtfgeneral}) to write the first summand in the above expression for $A$ as

\begin{align}
s_0\!\left(\ol{P_{g-1}}b_ga_g^{n+1}\upchi\right)&=s_0\!\left(\ol{P_{g-1}}\upchi\right)+\ol{P_{g-1}}s_0\!\left(b_ga_g^{n+1}\upchi\right)
=s_0(a_gU\upchi)+a_gU\left(\upbeta_g+b_g\sum_{i=0}^na_g^i\upalpha_g\right).\nonumber
\end{align}
The last equality uses an analogue of~(\ref{tqoiusazx}) together with the equality $a_gU=\ol{P_{g-1}}$.
Thus
\begin{align*}
A&=s_0(a_gU\upchi)+a_gUb_g\sum_{i=0}^na_g^i\upalpha_g-\upalpha_g+a_gUs_0(\ol{U}\upchi)-a_g\sum_{i=0}^n\ol{U}^{i-1}b_g\upalpha_g \\
&=s_0(a_gU\upchi)+a_gUb_g\sum_{i=0}^na_g^i\upalpha_g-\upalpha_g-a_gs_0(U\upchi)-a_gU\sum_{i=0}^n\ol{U}^ib_g\upalpha_g, \quad \mbox{by~(\ref{nataobs}),} \\
&=a_gUb_g\sum_{i=0}^na_g^i\upalpha_g-a_gU\sum_{i=0}^n\ol{U}^ib_g\upalpha_g, \quad \mbox{by~(\ref{laosjjsssjjss}),} \\
&=a_gUb_g\sum_{i=0}^na_g^i\upalpha_g-a_gU\sum_{i=0}^nb_ga_g^{i}\upalpha_g=0, \quad \mbox{by~(\ref{lemma32}),}
\end{align*}
as asserted.
\end{proof}

The verification of the last condition in~(\ref{contr4}) is just as easy as that of the first condition:
\begin{proposition}
The last condition~(\ref{contr4}) holds true.
\end{proposition}
\begin{proof}
Since $d_1d_2=0$, we have $d_2(s_1d_2-\id_{M_2^g}) = (d_2s_1)d_2-d_2=(\id_{M_1^g}-s_0d_1)d_2-d_2 = 0$,
so that $s_1d_2=\id_{M_2^g}$, as $d_2$ is injective.
\end{proof}

\section{Diagonal approximations and contracting homotopies}\label{tpanddapp}
Recall that, for $\uppi_g$-modules $A_i$, cup-products $H^*(\uppi_g;A_1)\otimes H^*(\uppi_g;A_2)\to H^*(\uppi_g;A_1\otimes A_2)$ can be assessed once a diagonal approximation $\Delta_*^g\colon M_*^g\to M_*^g\otimes M_*^g$ is made explicit. Attaining the corresponding goal for $\uppi_{g_1,\ldots,g_n}:=\uppi_{g_1}\times\cdots\times\uppi_{g_n}$ requires describing a diagonal approximation $\Delta_*^{g_1,\ldots,g_n}\colon M_*^{g_1,\ldots,g_n}\to M_*^{g_1,\ldots,g_n}\otimes M_*^{g_1,\ldots,g_n}$ where, as in Remark~\ref{aruonvcrujg}, $M_*^{g_1,\ldots,g_n}=M_*^{g_1}\otimes\cdots\otimes M_*^{g_n}$. With this in mind, we recall:

\begin{proposition}[{\cite[Proposition~3.1]{MR1200878}}]\label{alalyvyvnene}
Let $G$ be a group and
\begin{equation*}
\cdots\to X_{n+1}\stackrel{\delta_{n+1}\;}{\longrightarrow} X_n \stackrel{\delta_n}{\longrightarrow} \cdots\stackrel{\delta_3}{\longrightarrow} X_2\stackrel{\delta_2}{\longrightarrow} X_1 \stackrel{\delta_1}{\longrightarrow} X_0 \stackrel{\varepsilon}{\longrightarrow} \mathbb{Z}\to 0
\end{equation*}
be a free $G$-resolution of the trivial $G$-module $\mathbb{Z}$. For each $q\ge0$ let $B_q$ be a $G$-basis for $X_q$, and assume that $\varepsilon(b_0)=1$ for each $b_0\in B_0$. Consider the tensor square resolution $X\otimes X$ of $\,\mathbb{Z}=:(X\otimes X)_{-1}$. If $U=\{U_q\colon(X\otimes X)_q\to(X\otimes X)_{q+1}\}_{q\geq-1}$ is a contracting homotopy for
\begin{equation*}
\cdots\to (X\otimes X)_n \to \cdots\to (X\otimes X)_2\to (X\otimes X)_1 \to (X\otimes X)_0 \to \mathbb{Z}\to 0,
\end{equation*}
then a diagonal approximation $\Delta\colon X\to X\otimes X$ is determined inductively on basis elements $b_q\in B_q$ by setting $\Delta_0(b_0)=b_0\otimes b_0$ and, for $q>0$, $\Delta_q(b_q)=U_{q-1}\circ\Delta_{q-1}\circ\delta_{q}(b_q)$. 
\end{proposition}

Proposition~\ref{alalyvyvnene} allows us to translate the task of describing a formula for a diagonal approximation $\Delta_*^{g_1,\ldots,g_n}\colon M_*^{g_1,\ldots,g_n}\to M_*^{g_1,\ldots,g_n}\otimes M_*^{g_1,\ldots,g_n}$  into describing an explicit contracting homotopy for the $G$-resolution $M^{g_1,\ldots, g_n}_*\otimes M^{g_1,\ldots, g_n}_*$ of $\mathbb{Z}$.

\begin{lemma}\label{qqnnzzhysfbut}
For $1\le i\le n$, let $f_i,g_i\colon (K_i,\delta_i)\to (K'_i,\delta'_i)$ be maps of chain complexes. Assume there are chain homotopies $s_i\colon f_i\simeq g_i$ (meaning $\delta'_i s_i+s_i\delta_i=f_i-g_i$). Then the tensor maps $$f=\!\!\bigotimes_{1\le i\le n}\!\! f_i,\;g=\!\!\bigotimes_{1\le i\le n} \!\!g_i\colon\! \bigotimes_{1\le i\le n} \!\!(K_i,\delta_i)\to \bigotimes_{1\le i\le n}\!\!(K'_i,\delta'_i)$$ are chain homotopic through $s=\sum_{i=1}^ng_1\otimes\cdots\otimes g_{i-1}\otimes s_i\otimes f_{i+1}\otimes\cdots\otimes f_n\colon f\simeq g$. Explicitly, by the standard sign convention for the tensor product of chain maps (see the final paragraph in~\cite[Section~V.2]{MR672956}),
$$
s(\kappa_1\otimes\cdots\otimes\kappa_n)=\sum_{i=1}^n (-1)^{d_i} g_1(\kappa_1)\otimes\cdots\otimes g_{i-1}(\kappa_{i-1})\otimes s_i(\kappa_i)\otimes f_{i+1}(\kappa_{i+1})\otimes\cdots\otimes f_n(\kappa_n),
$$
where $d_i$ is the degree of $\kappa_1\otimes\cdots\otimes\kappa_{i-1}$.
\end{lemma}
\begin{proof}
Use induction and Proposition~9.1 in Chapter V of~\cite{MR0349792}.
\end{proof}

We use Lemma~\ref{qqnnzzhysfbut} in combination with the following observations. For a $G$-resolution
\begin{equation}\label{laresol}
\cdots\stackrel{\partial_3}\longrightarrow F_2\stackrel{\partial_2}\longrightarrow F_1\stackrel{\partial_1}\longrightarrow F_0\stackrel{\varepsilon}\longrightarrow\mathbb{Z}\to0,
\end{equation}
let $F_*$ denote the ``deleted'' sequence $\;\cdots\stackrel{\partial_3}\longrightarrow F_2\stackrel{\partial_2}\longrightarrow F_1\stackrel{\partial_1}\longrightarrow F_0\to 0$. Choose a $\mathbb{Z}$-map $s_{-1}\colon\mathbb{Z}\to F_0$ such that $\varepsilon\circ s_{-1}$ is the identity on $\mathbb{Z}$. It is transparent that the composite $s_{-1}\circ\varepsilon\colon F_0\to F_0$ extends to a map $s_{-1}\circ\varepsilon\colon F_*\to F_*$ of chain complexes vanishing on elements of positive dimension. Further, as observed in the proof of~\cite[Proposition~3.2]{MR1200878}, a collection of maps $s=\{s_i\colon F_i\to F_{i+1}\colon i\geq0\}$ assembles a homotopy $s\colon 1_{F_*}\simeq s_{-1}\circ\varepsilon$ if and only if the ``augmented'' collection $\{s_i\colon i\ge-1\}$ is a contracting homotopy for~(\ref{laresol}). In particular, the contracting homotopy~(\ref{hglksaytpowq}) constructed in the previous section yields explicit homotopies $s^{g_i}\colon 1_{M^{g_i}_*}\simeq s_{-1}^{g_i}\circ\epsilon^{g_i}$, and Lemma~\ref{qqnnzzhysfbut} implies:

\begin{corollary}\label{porfinlascontracciones} Let $\uppi_{g_1,\ldots,g_n}=\uppi_{g_1}\times\cdots\times\uppi_{g_n}$.
A contracting homotopy $u=\{u_k\colon k\ge-1\}$ for the (tensor) $\uppi_{g_1,\ldots,g_n}$-resolution $M_{*}^{g_1,\ldots, g_{n}}$ of $\mathbb{Z}$ is determined by the additive maps $u_{-1}:\z\to M_0^{g_1,\ldots, g_n}$ and $u_{k}:M_k^{g_1,\ldots, g_n}\to M^{g_1,\ldots, g_n}_{k+1}$, $k\ge0$, given by $\,u_{-1}(1)=\upchi^{g_1}\otimes \cdots \otimes \upchi^{g_n}$ and, for $k\ge0$, 
$$
u_k(y_1 \otimes \cdots \otimes y_n) = \sum_{i=1}^n 
\left( s_{-1}^{g_1}\circ\epsilon\right)(y_1)\otimes\cdots\otimes 
\left( s_{-1}^{g_{i-1}}\circ\epsilon\right)(y_{i-1})
\otimes s^{g_i}(y_i)
\otimes y_{i+1}\otimes\cdots\otimes y_n.
$$
\end{corollary}

Note that we have omitted the use of subscripts for chain maps and chain homotopies appearing on the right-hand side of the equality for $u_k(y_1\otimes\cdots\otimes y_n)$. This will also be the case in the expression for $w_k(z_1 \otimes z_2)$ in Corollary~\ref{porfinlascontraccioneseneltensor} below. Note also that the signs ``$(-1)^{d_i}$'' appearing in the formula for $s(\kappa_1\otimes\cdots\otimes\kappa_n)$ in Lemma~\ref{qqnnzzhysfbut} are not needed in Corollary~\ref{porfinlascontracciones}. In fact, the $i$-th summand in the expression for $u_k(y_1\otimes\cdots\otimes y_n)$ is non-zero only when $y_1\otimes\cdots\otimes y_{i-1}$ is of degree zero. Similar considerations yield:

\begin{corollary}\label{porfinlascontraccioneseneltensor}
A contracting homotopy $w=\{w_k\colon k\ge-1\}$ for the (diagonal) $\uppi_{g_1,\ldots,g_n}$-resolution $M_{*}^{g_1,\ldots, g_{n}}\otimes M_{*}^{g_1,\ldots, g_{n}}\to\mathbb{Z}$ is determined by additive morphisms $w_{-1}:\z\to M_0^{g_1,\ldots, g_n}\otimes M_0^{g_1,\ldots, g_n}$ and $w_{k}:\left(M_*^{g_1,\ldots, g_n}\otimes M_*^{g_1,\ldots, g_n}\right)_k\to \left(M_*^{g_1,\ldots, g_n}\otimes M_*^{g_1,\ldots, g_n}\right)_{k+1}$, $k\ge0$, given by $w_{-1}(1)=\left(\upchi^{g_1}\otimes\cdots\otimes\upchi^{g_n}\right)\otimes\left(\upchi^{g_1}\otimes\cdots\otimes\upchi^{g_n}\right)$ and, for $k\ge0$, $$w_k(z_1 \otimes z_2) = u(z_1)\otimes z_2+(u_{-1}\circ\epsilon)(z_1)\otimes u(z_2).$$
\end{corollary}

\section{Multiplicative structure in $H^*(\uppi_g;\widetilde{\mathbb{Z}})$}\label{seccionunfactor}
In this section we deal with $\uppi_g$-modules $\widetilde{\mathbb{Z}}$ whose underlying additive structure is~$\mathbb{Z}$. A generator $a_i$ or $b_i$ of $\uppi_g$ acts as multiplication by either 1 or $-1$ in any such $\widetilde{\mathbb{Z}}$. Therefore, $\widetilde{\mathbb{Z}}$ is completely characterized by specifying the subset $S\subseteq\{a_i,b_i\colon 1\le i\le g\}$ of generators that act by multiplication by $-1$. Since the defining relation in~(\ref{presentacion}) is a product of commutators, any subset $S$ is possible, and we use the more specific notation $\mathbb{Z}_S$ (instead of $\widetilde{\mathbb{Z}}$). For instance, $\mathbb{Z}_\varnothing$ stands for the trivial $\uppi_g$-module $\mathbb{Z}$. Note that $\mathbb{Z}_{S_1}\otimes\mathbb{Z}_{S_2}=\mathbb{Z}_{S_1\ominus S_2}$, where the tensor product ---taken over the integers--- is seen as a diagonal $\uppi_g$-module, and $S_1\ominus S_2:=S_1\cup S_2-S_1\cap S_2$ is the symmetric difference of $S_1$ and~$S_2$.

\medskip
Cup product maps
\begin{equation}\label{productosmezclados}
H^*(\uppi_g;\mathbb{Z}_{S_1})\otimes H^*(\uppi_g;\mathbb{Z}_{S_2})\to H^*(\uppi_g;\mathbb{Z}_{S_1\ominus S_2})
\end{equation}
with either $S_1=S_2$ or with some $S_i$ being empty are described (in a slightly indirect way) in~\cite[Theorem~3.5]{MR3386228} by taking advantage of the fact that any two non-trivial $\uppi_g$-modules $\widetilde{\mathbb{Z}}$ are isomorphic. In this section we work out the general form of~(\ref{productosmezclados}).

\medskip
Routine calculations based on Corollary~\ref{alalyvyvnene} and on the case $n=1$ in Corollary~\ref{porfinlascontraccioneseneltensor} yield the following description of a diagonal approximation $M_*^g\to M_*^g\otimes M_*^g$:

\begin{proposition}\label{exa1}
A diagonal approximation $\Delta=\Delta_*^g\colon M_*^g\to M_*^g\otimes M_*^g$ is given by
$\Delta(\upchi)=\upchi\otimes\upchi$, $\Delta(\upalpha_i)=\upalpha_i\otimes a_i\upchi+\upchi\otimes\upalpha_i$, $\,\Delta(\upbeta_i)=\upbeta_i\otimes b_i\upchi+\upchi\otimes\upbeta_i$ and
\begin{align*}
\Delta(\upomega)=&\sum_{i=1}^{g-1}\left( \sum_{j=1}^{i}P_{j-1}\left((1-c_jb_j)\upalpha_j+(a_j-c_j)\upbeta_j \rule{0mm}{5mm}\right)\!\right) \!\otimes P_{i}\left( \upalpha_{i+1}-b_i\upalpha_i+a_{i+1}\upbeta_{i+1}-\upbeta_i \rule{0mm}{4mm} \right) \\
&+\sum_{i=1}^{g}\left( P_{i-1}\upalpha_i \otimes P_{i-1}a_i\upbeta_i - P_i\upbeta_i\otimes P_i b_i \upalpha_i\rule{0mm}{5mm}\right) +\upchi\otimes\upomega+\upomega\otimes b_ga_g\upchi.
\end{align*}
\end{proposition}

Each commutator $c_i=[a_i,b_i]$ acts trivially on any $\mathbb{Z}_{S}$. Therefore, for the purposes of applying the functor $\Hom_{\uppi_g}(-;\mathbb{Z}_{S})$, the differential~(\ref{d2}) and the expression above for $\Delta(\upomega)$ can be taken to be, respectively, $d_2(\upomega)=\sum_{i=1}^g\left(\rule{0mm}{4mm}(1-b_i)\upalpha_i+(a_i-1)\upbeta_i\right)$ and

\begin{align}
\Delta(\upomega)&=\sum_{i=1}^{g-1} \sum_{j=1}^i\left((1-b_j)\upalpha_j+(a_j-1)\upbeta_j \rule{0mm}{5mm}\right) \!\otimes \left(\upalpha_{i+1}-b_i\upalpha_i+a_{i+1}\upbeta_{i+1}-\upbeta_i \rule{0mm}{4mm}\right)\nonumber\\
&\;\;\;\,+\sum_{i=1}^g \left( \rule{0mm}{4mm}\upalpha_i\otimes a_i\upbeta_i-\upbeta_i\otimes b_i\upalpha_i\right)+\upchi\otimes\upomega+\upomega\otimes b_ga_g\upchi.\label{diagonalita}
\end{align}

The description of the differential $d_*$ in Proposition~\ref{spelledout} together with the above considerations show that the cochain complex $M^*_g:=\Hom_{\uppi_g}(M^g_*,\mathbb{Z}_{S})$ has a graded $\mathbb{Z}$-basis consisting of the elements $\upchi^*$ in dimension 0, $\upalpha_i^*$ and $\upbeta_i^*$ ($1\leq i\leq g$) in dimension 1, and $\upomega^*$ in dimension 2, with (co)differentials
\begin{equation}\label{ytrcnvmbn}
d_1^*(\upchi^*)=2\upsigma^*,\qquad d_2^*(\upalpha_i^*)=\begin{cases}2\upomega^*, & b_i\in S,\\0, & b_i\not\in S,\end{cases}
\qquad d_2^*(\upbeta_i^*)=\begin{cases}-2\upomega^*, & a_i\in S,\\0, & a_i\not\in S,\end{cases}
\end{equation}
where $\upsigma^*$ is the sum of the duals of the greek-letter versions of the elements in $S$:
$$
\upsigma^*=\sum_{a_i\in S}\upalpha^*_i+\sum_{b_i\in S}\upbeta^*_i.
$$
Note that, as in Lemma~\ref{qqnnzzhysfbut}, we follow standard sign conventions (see~\cite[Eq.~1.6, p.~57]{MR672956}) in the first equality of~(\ref{ytrcnvmbn}).

\begin{example}\label{hhaayyttsmmmldksju}
The above considerations recover the usual fact that $H^*(\uppi_g;\mathbb{Z}_\varnothing)$ is torsion free with $\mathbb{Z}$-basis given by (the classes of\hspace{.3mm}) $\upchi^*$ in dimension zero, $\upalpha^*_i$, $\upbeta^*_i$ ($1\leq i\leq g$) in dimension~1, and $\upomega^*$ in dimension~2, and that the only non-zero cup products involving 1-dimensional classes are $$\upbeta^*_i\smile \upalpha^*_i=\upomega^*=-\upalpha^*_i\smile \upbeta^*_i$$ for $1\leq i\leq g$. More generally, if $M$ is a trivial $\uppi_g$-module, and $u,v\in\Hom_{\uppi_g}(M^g_1,M)$ are 1-cocycles, the cup product $[u]\smile[v]\in H^2(\uppi_g;M\otimes M)$ is represented by the cocycle
$$
\upomega\mapsto\sum_{i=1}^g\left(u(\upbeta_i)\otimes v(\upalpha_i)-u(\upalpha_i)\otimes v(\upbeta_i)\rule{0mm}{4mm}\right).
$$
The latter assertion is of course~\cite[Corollary~2.3]{MR3386228}, except that here we use the standard sign convention in the definition of cup products (see the last paragraph of Chapter V.2, and the third paragraph of Chapter V.3 in~\cite{MR672956}).
\end{example}

The analysis of the general case in~(\ref{productosmezclados}) is just as straightforward as that in Example~\ref{hhaayyttsmmmldksju}. We start by describing explicit cocycles generating each group $H^*(\uppi_g;\mathbb{Z}_S)$. Throughout the rest of the section, and unless it is explicitly noted otherwise, we assume $S\neq\varnothing$. 

\medskip An immediate consequence of~(\ref{ytrcnvmbn}) is:
\begin{corollary}\label{muyinmediato}
For $S\neq\varnothing$, $H^0(\uppi_g;\mathbb{Z}_S)=0$ while $H^2(\uppi_g;\mathbb{Z}_S)=\upomega^*\cdot\mathbb{Z}_2$, the group with two elements generated by (the cohomology class of) $\upomega^*$.
\end{corollary}

Slightly more elaborate is to spell out (in Corollaries~\ref{corodefix1} and~\ref{corodefix} below) the group structure of $H^1(\uppi_g;\mathbb{Z}_S)$. Start by considering the partition of $\{1,\ldots,g\}$ by the sets
\begin{align*}
Y&=\{ i\in\{1,\ldots,g\}\colon a_i\in S\mbox{ and }b_i\in S\},\\
A&=\{ i\in\{1,\ldots,g\}\colon a_i\in S\mbox{ and }b_i\not\in S\},\\
B&=\{ i\in\{1,\ldots,g\}\colon a_i\not\in S\mbox{ and }b_i\in S\},\\
N&=\{ i\in\{1,\ldots,g\}\colon a_i\not\in S\mbox{ and }b_i\not\in S\}.
\end{align*}
We need to chose a ``pivot'' in $S$.

\medskip\noindent
{\bf Case A}: There is a generator $a_{i_0}\in S$. Consider the elements $\upsigma^*_i, \bar{\upalpha}^*_i,\bar{\upbeta}^*_i\in\ker(d_2^*)$ ($1\leq i\leq g$) defined through the table:
\begin{center}
  \begin{tabular}{ | c| c | c | c | }
    \hline
    \rule{0mm}{5mm} \raisebox{.3mm}{$i\,$ belongs to} & \raisebox{.5mm}{$\upsigma^*_i$} & \raisebox{.5mm}{$\bar{\upalpha}^*_i$} & \raisebox{.5mm}{$\bar{\upbeta}^*_i$} \\ \hline
    \rule{0mm}{5mm}$Y$ & \raisebox{.7mm}{$\upalpha^*_i+\upbeta^*_i$} & \raisebox{.3mm}{$0$} & \raisebox{.7mm}{$\upbeta^*_i-\upbeta^*_{i_0}$} \\ \hline
    \rule{0mm}{5mm}$A$ & \raisebox{.7mm}{$\upalpha^*_i$} & \raisebox{.3mm}{$0$} & \raisebox{.7mm}{$\upbeta^*_i-\upbeta^*_{i_0}$} \\ \hline
    \rule{0mm}{5mm}$B$ & \raisebox{.7mm}{$\upbeta^*_i$} & \raisebox{.7mm}{$\upalpha^*_i+\upbeta^*_{i_0}$} & \raisebox{.3mm}{$0$} \\ \hline
    \rule{0mm}{5mm}$N$ & \raisebox{.3mm}{$0$} & \raisebox{.7mm}{$\upalpha^*_i$} & \raisebox{.7mm}{$\upbeta^*_i$} \\ \hline
  \end{tabular}
\end{center}
Note that $\upsigma^*=\sum_i\upsigma^*_i$ and $\bar{\upbeta}^*_{i_0}=0$ ($i_0\in Y\cup A$). A direct linear-algebra calculation yields:

\begin{corollary}\label{corodefix1}
Fix a generator $a_{i_0}\in S$ (if there is any) and let $H_{\upsigma}$, $H_{\upalpha}$ and $H_{\upbeta}$ denote the free abelian groups with bases given by the (cohomology classes of the) elements indicated in the table: 
\begin{center}
  \begin{tabular}{ | c| c | c | c | }
    \hline
    \rule{0mm}{5mm}\raisebox{.7mm}{group} & \raisebox{.3mm}{$H_{\upsigma}$} & \raisebox{.3mm}{$H_{\upalpha}$} & \raisebox{.3mm}{$H_{\upbeta}$} \\ \hline
    \rule{0mm}{5mm}\raisebox{.5mm}{basis} & \raisebox{.5mm}{$\{\upsigma^*_i\colon i\in Y\cup A\cup B$\}} & \raisebox{.5mm}{$\{\bar{\upalpha}^*_i\colon i\in B\cup N$\}} & \raisebox{.5mm}{$\{\bar{\upbeta}^*_i\colon i\in Y\cup A\cup N-\{i_0\}$\}} \\ \hline  \end{tabular}
\end{center}
Then $H^1(\uppi_g;\mathbb{Z}_S)=H_{\upsigma}/2D\oplus H_{\upalpha}\oplus H_{\upbeta}\cong\mathbb{Z}^{2g-2}\oplus\mathbb{Z}_2$, where $D$ stands for the ``diagonal'' subgroup of $H_{\upsigma}$ generated by $\upsigma^*$. The element of $2$-torsion is (the cohomology class of) $\upsigma^*$, while a basis for the summand $\mathbb{Z}^{2g-2}$ is given by the (cohomology classes of the) cocycles
\begin{itemize}
\item $\upsigma^*_i$ with $i\in Y\cup A\cup B-\{i_0\}$;
\item $\bar{\upalpha}^*_i$ with $i\in B\cup N$;
\item $\bar{\upbeta}^*_i$ with $i\in Y\cup A\cup N-\{i_0\}$.
\end{itemize}
\end{corollary}

\medskip\noindent
{\bf Case B}: There is a generator $b_{i_0}\in S$. Consider the elements $\upsigma^*_i, \bar{\upalpha}^*_i,\bar{\upbeta}^*_i\in\ker(d_2^*)$ ($1\leq i\leq g$) defined through the table:
\begin{center}
  \begin{tabular}{ | c| c | c | c | }
    \hline
    \rule{0mm}{5mm} \raisebox{.3mm}{$i\,$ belongs to} & \raisebox{.5mm}{$\upsigma^*_i$} & \raisebox{.5mm}{$\bar{\upalpha}^*_i$} & \raisebox{.5mm}{$\bar{\upbeta}^*_i$} \\ \hline
    \rule{0mm}{5mm}$Y$ & \raisebox{.7mm}{$\upalpha^*_i+\upbeta^*_i$} & \raisebox{.7mm}{$\upalpha^*_i-\upalpha^*_{i_0}$} & \raisebox{.3mm}{$0$} \\ \hline
    \rule{0mm}{5mm}$B$ & \raisebox{.7mm}{$\upbeta^*_i$} & \raisebox{.7mm}{$\upalpha^*_i-\upalpha^*_{i_0}$} & \raisebox{.3mm}{$0$} \\ \hline
    \rule{0mm}{5mm}$A$ & \raisebox{.7mm}{$\upalpha^*_i$} & \raisebox{.3mm}{$0$} & \raisebox{.7mm}{$\upbeta^*_i+\upalpha^*_{i_0}$} \\ \hline
    \rule{0mm}{5mm}$N$ & \raisebox{.3mm}{$0$} & \raisebox{.7mm}{$\upalpha^*_i$} & \raisebox{.7mm}{$\upbeta^*_i$} \\ \hline
  \end{tabular}
\end{center}
Again $\upsigma^*=\sum_i\upsigma^*_i$, but now $\bar{\upalpha}^*_{i_0}=0$ ($i_0\in Y\cup B$).

\begin{corollary}\label{corodefix}
Fix a generator $b_{i_0}\in S$ (if there is any) and let $H_{\upsigma}$, $H_{\upalpha}$ and $H_{\upbeta}$ denote the free abelian groups with bases given by the (cohomology classes of the) elements indicated in the table:
\begin{center}
  \begin{tabular}{ | c| c | c | c | }
    \hline
    \rule{0mm}{5mm}\raisebox{.7mm}{group} & \raisebox{.3mm}{$H_{\upsigma}$} & \raisebox{.3mm}{$H_{\upalpha}$} & \raisebox{.3mm}{$H_{\upbeta}$} \\ \hline
    \rule{0mm}{5mm}\raisebox{.5mm}{basis} & \raisebox{.5mm}{$\{\upsigma^*_i\colon i\in Y\cup A\cup B$\}} & \raisebox{.5mm}{$\{\bar{\upalpha}^*_i\colon i\in Y\cup B\cup N-\{i_0\}$\}} & \raisebox{.5mm}{$\{\bar{\upbeta}^*_i\colon i\in A\cup N$\}} \\ \hline  \end{tabular}
\end{center}
Then $H^1(\uppi_g;\mathbb{Z}_S)=H_{\upsigma}/2D\oplus H_{\upalpha}\oplus H_{\upbeta}\cong\mathbb{Z}^{2g-2}\oplus\mathbb{Z}_2$, where $D$ stands for the ``diagonal'' subgroup of $H_{\upsigma}$ generated by $\upsigma^*$. The element of $2$-torsion is (the cohomology class of) $\upsigma^*$, while a basis for the summand $\mathbb{Z}^{2g-2}$ is given by the (cohomology classes of the) cocycles
\begin{itemize}
\item $\upsigma_i^*$ with $i\in Y\cup A\cup B-\{i_0\}$;
\item $\bar{\upalpha}_i^*$ with $i\in Y\cup B\cup N-\{i_0\}$;
\item $\bar{\upbeta}_i^*$ with $i\in A\cup N$.
\end{itemize}
\end{corollary}

Having described explicit cocycles representing basis elements of $H^*(\uppi_g;\mathbb{Z}_S)$, all that remains to do in order to fully describe the cup-product maps~(\ref{productosmezclados}) is to get a complete formula for the corresponding cochain-level products
\begin{equation}\label{cochainproduct}
\Hom_{\uppi_g}(M_*^g,\mathbb{Z}_{S_1})\otimes\Hom_{\uppi_g}(M_*^g,\mathbb{Z}_{S_2})\to\Hom_{\uppi_g}(M_*^g,\mathbb{Z}_{S_1\ominus S_2}).
\end{equation}
The required formula, spelled out in Lemma~\ref{cochain-level cup-product} below, follows from Proposition~\ref{exa1} and~(\ref{diagonalita}). Here we use the notation $\uplambda$ and $\ell$ to mean $\ell=a_i$ provided $\uplambda=\upalpha_i$, while $\ell=b_i$ provided $\uplambda=\upbeta_i$ ($1\leq i\leq g$).

\begin{lemma}\label{cochain-level cup-product}
The cochain-level cup-product map~(\ref{cochainproduct}) is determined by the relations $\upchi^*\cdot\Theta^*=\Theta^*$ for $\Theta^*\in\{\upchi^*, \uplambda^*, \upomega^*\}$, $\,\uplambda_r^*\cdot\upmu_s^*=\delta(\uplambda_r,\upmu_s)\upomega^*$, \,and
\begin{align*}
\uplambda^*\cdot\upchi^*=&\begin{cases}-\uplambda^*,&\ell\in S_2; \\ \phantom{-}\uplambda^*,&\mbox{otherwise},\end{cases}\\
\upomega^*\cdot\upchi^*=&\begin{cases}\phantom{-}\upomega^*,& a_g,b_g\in S_2 \mbox{ \; or \;} a_g,b_g\not\in S_2;\\-\upomega^*,& \mbox{otherwise},\end{cases}
\end{align*}
where
\begin{align*}
\delta(\upalpha_r,\upalpha_s)=&
\left\{ \begin{matrix*}[l] -2, & r<s\; \mbox{ and \,} b_r\in S_1 \\ \phantom{-}0, & \mbox{otherwise} \end{matrix*} \right\}+
\left\{ \begin{matrix*}[l] \phantom{-}2, & r\leq s<g, \; b_r\in S_1\; \mbox{ and \,} b_s      \in S_2\\ 
                                                     -2, & r\leq s<g, \; b_r\in S_1\; \mbox{ and \,} b_s\not\in S_2\\
                                    \phantom{-}0, & \mbox{otherwise}  \end{matrix*} \right\},\\
\delta(\upalpha_r,\upbeta_s)=&
\left\{ \begin{matrix*}[l]                    -2, & r< s, \; b_r\in S_1\; \mbox{ and \,} a_s      \in S_2\\ 
                                      \phantom{-} 2, & r< s, \; b_r\in S_1\; \mbox{ and \,} a_s\not\in S_2\\
                                    \phantom{-}0, & \mbox{otherwise}  \end{matrix*} \right\}+
\left\{ \begin{matrix*}[l] 2, & r\leq s<g\; \mbox{ and \,} b_r\in S_1 \\ 0, & \mbox{otherwise} \end{matrix*} \right\}\\
&+\left\{ 
\begin{matrix*}[l]            \phantom{-}1, & r=s, \; \mbox{ and \,} a_r \in S_2\\ 
                                                      -1, & r=s, \; \mbox{ and \,} a_r\not\in S_2\\
                                    \phantom{-}0, & \mbox{otherwise}  \end{matrix*}\right\},\\
\delta(\upbeta_r,\upalpha_s)=& 
\left\{ \begin{matrix*}[l] 2, & r<s\; \mbox{ and \,} a_r\in S_1 \\ 0, & \mbox{otherwise} \end{matrix*}\right\}+
\left\{ \begin{matrix*}[l]                    -2, & r\leq s<g, \; a_r\in S_1\; \mbox{ and \,} b_s \in S_2\\ 
                                      \phantom{-} 2, & r\leq s<g, \; a_r\in S_1\; \mbox{ and \,} b_s\not\in S_2\\
                                    \phantom{-}0, & \mbox{otherwise}  \end{matrix*} \right\}\\
&-\left\{ 
\begin{matrix*}[l]            \phantom{-}1, & r=s, \; \mbox{ and \,} b_r \in S_2\\ 
                                                      -1, & r=s, \; \mbox{ and \,} b_r\not\in S_2\\
                                    \phantom{-}0, & \mbox{otherwise}  \end{matrix*}
\right\},\\
\delta(\upbeta_r,\upbeta_s)=& 
\left\{ \begin{matrix*}[l] \phantom{-}2, & r<s, \; a_r\in S_1\; \mbox{ and \,} a_s \in S_2\\ 
                                                     -2, & r<s, \; a_r\in S_1\; \mbox{ and \,} a_s \not\in S_2\\
                                    \phantom{-}0, & \mbox{otherwise}  \end{matrix*} \right\}+
\left\{ \begin{matrix*}[l] -2, & r\leq s<g\; \mbox{ and \,} a_r\in S_1 \\ \phantom{-}0, & \mbox{otherwise} \end{matrix*} \right\}.
\end{align*}
\end{lemma}

\medskip By Corollary~\ref{muyinmediato}, $1\in H^0(\uppi_g,\mathbb{Z}_\varnothing)$ is the only non-trivial 0-dimensional cohomology class in the $\widetilde{\mathbb{Z}}$-twisted cohomology of $\uppi_g$. It therefore suffices to describe cup products~(\ref{productosmezclados}) of 1-dimensional classes. When $S_1\neq S_2$, such a product: 
\begin{enumerate}[(i)]
\item lies in a group isomorphic to~$\mathbb{Z}_2$, 
\item is therefore an $\varepsilon$-multiple ($\varepsilon\in\{0,1\}$) of the cohomology class of $\upomega^*$, and
\item can then be easily read off from Lemma~\ref{cochain-level cup-product} and Corollaries~\ref{corodefix1} and~\ref{corodefix}.
\end{enumerate}
Of course, ({\it iii\hspace{.4mm}}) holds for $S_1=S_2$ too, and we recover~\cite[Theorem~3.5]{MR3386228}, except for sign conventions, as in Example~\ref{hhaayyttsmmmldksju}.

\smallskip
The convenience of having a full description of cup products at the cochain level (Lemma~\ref{cochain-level cup-product}) will be clear in the next section, where we deal with cup-products in the $\widetilde{\mathbb{Z}}$-twisted cohomology of products of groups~$\uppi_g$. The latter is the information we need for the topological-robotics application at the end of the paper.

\section{Multiplicative structure in $H^*(\uppi_{g_1}\times\uppi_{g_2};\widetilde{\mathbb{Z}})$}\label{secciondosfactores}
We now address cup-products in the cohomology of a product $\uppi_{g_1,\ldots,g_n}:=\uppi_{g_1}\times\cdots\times\uppi_{g_n}$. For a $\uppi_{g_1,\ldots,g_n}$-module $\widetilde{\mathbb{Z}}$, let $S_i$ stand for the set of generators $a_j$ and $b_j$ of the $i$-th factor $\uppi_{g_i}$ that act on $\widetilde{\mathbb{Z}}$ by changing sign. Then $\widetilde{\mathbb{Z}}$ is the tensor product $\bigotimes_{i}\mathbb{Z}_{S_i}$ where each $\mathbb{Z}_{S_i}$ is the $\uppi_{g_1,\ldots,g_n}$-module obtained by restriction of scalars via the $i$-th projection map $\uppi_{g_1,\ldots,g_n}\to\uppi_{g_i}$. We thus use the notation $\mathbb{Z}_{S_1\otimes\cdots\otimes S_n}$ for such a $\uppi_{g_1,\ldots,g_n}$-module $\widetilde{\mathbb{Z}}=\bigotimes_{i}\mathbb{Z}_{S_i}$. 

\medskip
Note that the K\"unneth formula
$$
H^m(G\times G';M\otimes M')=\bigoplus_{p=0}^mH^p(G;M)\otimes H^{m-p}(G';M')\oplus\bigoplus_{p=0}^{m+1}\mbox{Tor}(H^p(G;M),H^{m+1-p}(G';M'))
$$
can be used to get an {\it additive} description of $H^*(\uppi_{g_1,\ldots,g_n};\mathbb{Z}_{S_1\otimes\cdots\otimes S_n})$. Our task, in order to get the {\it multiplicative} structure, is to describe (A) explicit cocycles representing a full set of cohomology generators, and (B) cup products at the level of cochains.

In principle, as in the previous section, ingredient (B) above would require describing diagonal approximations
\begin{equation}\label{ixicuvyvtbrb}
\Delta^{g_1,\ldots,g_n}\colon M_*^{g_1,\ldots,g_n}\to M_*^{g_1,\ldots,g_n}\otimes M_*^{g_1,\ldots,g_n}.
\end{equation}
These could be made explicit through Corollaries~\ref{alalyvyvnene} and~\ref{porfinlascontraccioneseneltensor}. But a more accessible alternative comes by observing that the (shuffled) tensor product of diagonal approximations yields a diagonal approximation for the tensor product (see~\cite[page 10, exercise~7]{MR672956}). For instance, for diagonal approximations $\Delta^{g_i}\colon M_*^{g_i}\to M_*^{g_i}\otimes M_*^{g_i}$ ($i=1,2$), the composite
\begin{equation}\label{prodtensor}
M_*^{g_1}\otimes M_*^{g_2}\stackrel{\Delta^{g_1}\otimes \Delta^{g_2}}{-\!\!\!-\!\!\!-\!\!\!-\!\!\!\longrightarrow}M_*^{g_1}\otimes M_*^{g_1}\otimes M_*^{g_2}\otimes M_*^{g_2}\stackrel{1\otimes T\otimes1}{-\!\!\!-\!\!\!-\!\!\!\longrightarrow}M_*^{g_1}\otimes M_*^{g_2}\otimes M_*^{g_1}\otimes M_*^{g_2}
\end{equation}
is a diagonal approximation for $M_*^{g_1,g_2}$, where $T$ is the chain map $T(x\otimes y)=(-1)^{\epsilon(x,y)} y\otimes x$, and $\epsilon(x,y)$ is the product of the degrees of $x$ and $y$. (Note that the standard sign convention does not introduce a sign associated to the first map in~(\ref{prodtensor}), for both tensor factors in that map have degree zero.)

The description of explicit diagonal approximations is not the only complexity issue to worry about. The process of dualizing any given diagonal approximation~(\ref{ixicuvyvtbrb}) ---in order to assess cup products--- tends to be tortuous and quickly becomes inaccessible, as tensor cochain complexes $M_*^{g_1,\ldots,g_n}$ have many types of generators. The most efficient alternative is to take full advantage of Lemma~\ref{cochain-level cup-product} and Lemma~\ref{signconventionfinal} below, a result whose proof is an easy exercise with standard sign conventions, e.g.~\cite[Chapter~I, Section~0]{MR672956}. Due to the applications we aim at, Lemma~\ref{signconventionfinal} is stated only for products with two factors. The reader will easily spell out and prove a version of Lemma~\ref{signconventionfinal} for more factors.

\begin{lemma}\label{signconventionfinal}
Consider the cochain complex isomorphism
$$
t\colon\Hom_{\uppi_{g_1}}(M_*^{g_1};\mathbb{Z}_{S_1})\otimes\Hom_{\uppi_{g_2}}(M_*^{g_2};\mathbb{Z}_{S_2})\to\Hom_{\uppi_{g_1,g_2}}(M_*^{g_1,g_2};\mathbb{Z}_{S_1\ot S_2})
$$
sending a tensor product $u\ot v$ of graded morphisms into the graded tensor product morphism $(-1)^{|u||v|}u\ot v\,$ (cf.~\cite[page~10, exercise~7]{MR672956}). Let $\Delta^{g_i}\colon M_*^{g_i}\to M_*^{g_i}\ot M_*^{g_i}$ $(i\in\{1,2\})$ be diagonal approximations with corresponding induced cochain products $$\Hom_{\uppi_{g_i}}(M_*^{g_i};\mathbb{Z}_{S_i})\otimes\Hom_{\uppi_{g_i}}(M_*^{g_i};\mathbb{Z}_{S'_i})\stackrel{p_i}\longrightarrow\Hom_{\uppi_{g_i}}(M_*^{g_i};\mathbb{Z}_{S_i\ominus S_i'}).$$ Then there is a commutative diagram

\begin{flushleft}
\begin{tikzpicture}[commutative diagrams/every diagram, scale=1, every node/.style={scale=0.74}]
        \matrix (m) [
            matrix of math nodes,
            row sep=2em,
            column sep=0.3em,
            text height=1.5ex, text depth=0.25ex
        ]
        {    \Hom_{\uppi_{g_{1},g_{2}}}(M_*^{g_{1},g_{2}};\mathbb{Z}_{S_1\ot S_2}) \otimes \Hom_{\uppi_{g_{1},g_{2}}}(M_*^{g_{1},g_{2}};\mathbb{Z}_{S'_1\ot S'_2}) & {} & \Hom_{\uppi_{g_{1},g_{2}}}(M_*^{g_{1},g_{2}};\mathbb{Z}_{(S_1\ominus S'_1)\ot(S_2\ominus S'_2)}) \\
             \Hom_{\uppi_{g_{1}}}(M_*^{g_{1}};\mathbb{Z}_{S_1}) \otimes \Hom_{\uppi_{g_{2}}}(M_*^{g_{2}};\mathbb{Z}_{S_2}) \otimes \Hom_{\uppi_{g_{1}}}(M_*^{g_{1}};\mathbb{Z}_{S'_1}) \otimes \Hom_{\uppi_{g_{2}}}(M_*^{g_{2}};\mathbb{Z}_{S'_2}) & {} & {} \\
             \Hom_{\uppi_{g_{1}}}(M_*^{g_{1}};\mathbb{Z}_{S_1}) \otimes \Hom_{\uppi_{g_{1}}}(M_*^{g_{1}};\mathbb{Z}_{S'_1}) \otimes \Hom_{\uppi_{g_{2}}}(M_*^{g_{2}};\mathbb{Z}_{S_2}) \otimes \Hom_{\uppi_{g_{2}}}(M_*^{g_{2}};\mathbb{Z}_{S'_2}) & {} & {}\\
             \Hom_{\uppi_{g_{1}}}(M_*^{g_{1}};\mathbb{Z}_{S_1\ominus S'_1}) \otimes \Hom_{\uppi_{g_{2}}}(M_*^{g_{2}};\mathbb{Z}_{S_2\ominus S'_2}) & {}  \\
            };
        \path[overlay,->, font=\scriptsize,>=latex]
        (m-1-1) edge node [above]{$p$} (m-1-3)
        (m-2-1) edge node [left]{$\cong$}(m-1-1)
        (m-2-1) edge node [right]{$t\ot t$}(m-1-1)
        (m-2-1) edge node [right]{$1\ot T\ot1$}(m-3-1)
        (m-3-1) edge node [right]{$p_1\ot p_2$} (m-4-1)        
        (m-4-1) edge[out=0,in=270] node [below]{$t$} (m-1-3);
\end{tikzpicture}
\end{flushleft}
whose top horizontal morphism is the cochain product induced by the shuffled diagonal approximation~(\ref{prodtensor}). 
\end{lemma}

Thus, the bulk of the work amounts to identifying explicit cocycles that represent generators of $H^*(\uppi_{g_1,\ldots,g_n};\mathbb{Z}_{S_1\otimes\cdots\otimes S_n})$. Note that, if some $S_i$ is empty, cocycle representatives for $\uppi_{g_1,\ldots g_n}$ can be obtained as exterior graded-tensor products of those for $\uppi_{g_1,\ldots,g_{i-1},g_{i+1},\ldots,g_n}$ and those for $\uppi_{g_i}$ (the latter ones have been described in Example~\ref{hhaayyttsmmmldksju}). Indeed, in such conditions, the corresponding $H^*(\uppi_i;\mathbb{Z}_{S_i})$ is torsion free, and the K\"unneth formula becomes 
$$
H^*(\uppi_{g_1,\ldots,g_n};\mathbb{Z}_{S_1\otimes\cdots\otimes S_n})=H^*(\uppi_{g_1,\ldots,g_{i-1},g_{i+1},\ldots,g_n};\mathbb{Z}_{S_1\otimes\cdots\otimes S_{i-1}\otimes S_{i+1}\otimes\cdots\otimes S_n})\otimes H^*(\uppi_{g_i};\mathbb{Z}_{S_i}).
$$
Consequently, it is safe to assume that $S_i\neq\varnothing$ for all $i\in\{1,\ldots,n\}$ (as done in Propositions~\ref{cdelpdim1}--\ref{cdelpdim3}). Details are worked out next for the case of $\uppi_{g_1,g_2}$.

\begin{proposition}\label{cdelpdim0}
If $S_1=\varnothing=S_2$, then $H^0(\uppi_{g_1,g_2};\mathbb{Z}_{S_1\otimes S_2})=\mathbb{Z}$ with generator represented by the cocycle $(\upchi\otimes\upchi)^*$. If either $S_1\neq\varnothing$ or $S_2\neq\varnothing$, then $H^0(\uppi_{g_1,g_2};\mathbb{Z}_{S_1\otimes S_2})=0$. 
\end{proposition}
\begin{proof}
This result is a direct consecuence of the K\"unneth formula and Corollary~\ref{muyinmediato}. We include a direct argument for future reference. Proposition~\ref{spelledout} implies that the differential $d_1^*\colon\Hom_{\uppi_{g_1,g_2}}(M_0^{g_1,g_2};\mathbb{Z}_{S_1\otimes S_2})\to\Hom_{\uppi_{g_1,g_2}}(M_1^{g_1,g_2};\mathbb{Z}_{S_1\otimes S_2})$ is determined by
\begin{equation}\label{diffi1}
d_1^*\left((\upchi\otimes\upchi)^*\right)=2\left(\, \sum_{\ell_i\in S_1}(\uplambda_i\otimes\upchi)^* + \sum_{m_j\in S_2}(\upchi\otimes\upmu_j)^*\right)
\end{equation}
(as in the case of $d_1^*(\upchi^*)$ in~(\ref{ytrcnvmbn}), this accounts for the standard sign convention for the coboundary map). So, $d_1^*=0$ for $S_1=S_2=\varnothing$, while $d_1^*$ is injective otherwise.
\end{proof}

Following the notation in Lemma~\ref{cochain-level cup-product}, we write $\uplambda_i$ in~(\ref{diffi1}) and below for either $\upalpha_i$ or~$\upbeta_i$ ($1\leq i\leq g_1$), depending on whether $\ell_i$ is $a_i$ or $b_i$, respectively. Likewise, $\upmu_j$ stands for the greek-letter version of a generator $m_j$ ($1\leq j\leq g_2$) of $\uppi_{g_2}$. Note that we are not using any special notation to tell apart the generators (or their greek-letter versions) of the two groups $\uppi_{g_i}$. This causes no confusion in~(\ref{diffi1}) as the meaning is implicit from the side the generators appear on a tensor product. If the tensor-side distinction is not available, we will use a functional notation such as $a_i(g_1), \ell_i(g_1)\in \uppi_{g_1}$ ($1\leq i\leq g_1$) and $\upbeta_j(g_2),\upmu_j(g_2)\in M_*^{g_2}$ ($1\leq j\leq g_2$). Nonetheless, the functional notation will be waived when making a distinction becomes irrelevant.

As suggested in the proof of Proposition~\ref{cdelpdim0},~(\ref{diffi1}) can be obtained by dualizing the tensor square of the differential $d_*$ in Proposition~\ref{spelledout}. The same formula can be obtained with much less effort from~(\ref{ytrcnvmbn}) and Lemma~\ref{signconventionfinal}. The latter is the method we use for formulas~(\ref{diffi2A})--(\ref{diffi4B}) below. The straightforward details are left as an exercise for the reader.

\begin{lemma}
The differential $d_2^*\colon\Hom_{\uppi_{g_1,g_2}}(M_1^{g_1,g_2};\mathbb{Z}_{S_1\otimes S_2})\to\Hom_{\uppi_{g_1,g_2}}(M_2^{g_1,g_2};\mathbb{Z}_{S_1\otimes S_2})$ is determined by
\begin{align}
d^*_2\left((\uplambda_i\otimes\upchi)^*\right)&=2\left(\sum_{m_j\in S_2}(\uplambda_i\otimes\upmu_j)^*+\epsilon_{\uplambda_i}\left\{\begin{matrix}
(\upomega\otimes\upchi)^*,&\mbox{if \ $\widehat{\,\ell_i\,}\in S_1$}\\ 0,&\mbox{if \ $\widehat{\,\ell_i\,}\not\in S_1$} 
\end{matrix}\right\}\right),\label{diffi2A}\\
d^*_2\left((\upchi\otimes\upmu_j)^*\right)&=2\left(-\sum_{\ell_i\in S_1}(\uplambda_i\otimes\upmu_j)^*+\epsilon_{\upmu_j}\left\{\begin{matrix}
(\upchi\otimes\upomega)^*,&\mbox{if \ $\widehat{m_j}\in S_2$}\\ 0,&\mbox{if \ $\widehat{m_j}\not\in S_2$} 
\end{matrix}\right\}\right).\label{diffi2B}
\end{align}
Here $\epsilon_{\upalpha_k}=1$, $\epsilon_{\upbeta_k}=-1$, $\widehat{a_k}=b_k$ and $\widehat{b_k}=a_k$ for $k=1,\ldots,g_i$ and $i=1,2$. Likewise, the differential $d_3^*\colon\Hom_{\uppi_{g_1,g_2}}(M_2^{g_1,g_2};\mathbb{Z}_{S_1\otimes S_2})\to\Hom_{\uppi_{g_1,g_2}}(M_3^{g_1,g_2};\mathbb{Z}_{S_1\otimes S_2})$ is determined by
\begin{align}
d^*_3\left((\upomega\otimes\upchi)^*\right)&=2\sum_{m_j\in S_2}(\upomega\otimes\upmu_j)^*,\label{diffi3A}\\
d^*_3\left((\upchi\otimes\upomega)^*\right)&=2\sum_{\ell_i\in S_1}(\uplambda_i\otimes\upomega)^*,\label{diffi3B}\\
d^*_3\left((\uplambda_i\otimes\upmu_j)^*\right)&=2\left(\left\{\begin{matrix}
\epsilon_{\upmu_j}(\uplambda_i\otimes\upomega)^*,&\mbox{if \ $\widehat{m_j}\in S_2$}\\ 0,&\mbox{if \ $\widehat{m_j}\not\in S_2$} 
\end{matrix}\right\}-\left\{\begin{matrix}
\epsilon_{\uplambda_i}(\upomega\otimes\upmu_j)^*,&\mbox{if \ $\widehat{\,\ell_i\,}\in S_1$}\\ 0,&\mbox{if \ $\widehat{\,\ell_i\,}\not\in S_1$}\end{matrix}\right\}\right),\label{diffi3C}
\end{align}
and the differential $d_4^*\colon\Hom_{\uppi_{g_1,g_2}}(M_3^{g_1,g_2};\mathbb{Z}_{S_1\otimes S_2})\to\Hom_{\uppi_{g_1,g_2}}(M_4^{g_1,g_2};\mathbb{Z}_{S_1\otimes S_2})$ by
\begin{align}
d_4^*((\uplambda_i\otimes\upomega)^*)= & \begin{cases} 2\epsilon_{\uplambda_i}(\upomega\otimes\upomega)^*, & \widehat{\,\ell_i\,}\in S_1; \\ 0, & \widehat{\,\ell_i\,}\not\in S_1,
\end{cases}\label{diffi4A}\\
d_4^*((\upomega\otimes\upmu_j)^*)= & \begin{cases} 2\epsilon_{\upmu_j}(\upomega\otimes\upomega)^*, & \widehat{m_j}\in S_2; \\ 0, & \widehat{m_j}\not\in S_2.
\end{cases}\label{diffi4B}
\end{align}
\end{lemma}

\begin{proposition}\label{cdelpdim1}
For $S_1\neq\varnothing\neq S_2$, $H^1(\uppi_{g_1,g_2};\mathbb{Z}_{S_1\otimes S_2})=\mathbb{Z}_2$. The generator is represented by the cocycle
\begin{equation}\label{elciclitoendim1}
\upnu=\sum_{\ell_i\in S_1}(\uplambda_i\otimes\upchi)^* + \sum_{m_j\in S_2}(\upchi\otimes\upmu_j)^*.
\end{equation}
\end{proposition}
\begin{proof}
The first assertion follows from the K\"unneth formula together with Corollaries~\ref{muyinmediato}, \ref{corodefix1} and~\ref{corodefix}. The fact that~(\ref{elciclitoendim1}) is a representing cocycle for the generator follows from a direct verification using~(\ref{diffi1})--(\ref{diffi2B}) or, more easily, from~(\ref{diffi1}) and the fact that $d_2^* d_1^*=0$, as $\Hom_{\uppi_{g_1,g_2}}(M_*^{g_1,g_2};\mathbb{Z}_{S_1\otimes S_2})$ is (torsion) free.
\end{proof}

Proposition~\ref{cdelpdim1} identifies the first dimension in the twisted cohomology of $\uppi_{g_1,g_2}$ with classes that are not given as exterior products (i.e., classes that come from the Tor part in the K\"unneth formula). Such classes appear also in dimension 2 and 3, in view of Propositions~\ref{cdelpdim2} and~\ref{cdelpdim3} below (but not in dimension 4, as observed in Remark~\ref{quepasaendimension4}).

\begin{proposition}\label{cdelpdim2}
For $S_1\neq\varnothing\neq S_2$, $H^2(\uppi_{g_1,g_2};\mathbb{Z}_{S_1\otimes S_2})=F^{g_1,g_2}_{S_1\otimes S_2}\oplus T^{g_1,g_2}_{S_1\otimes S_2}$, where $F^{g_1,g_2}_{S_1\otimes S_2}$ is free abelian of rank $4(g_1-1)(g_2-1)$, and $T^{g_1,g_2}_{S_1\otimes S_2}$ is a $\mathbb{Z}_2$-vector space of dimension $2(g_1+g_2)-1$. Cocycle representatives for a $\mathbb{Z}$-basis of $F^{g_1,g_2}_{S_1\otimes S_2}$ are given by exterior products of (cocycle representatives of\hspace{.5mm}) the basis elements described in Corollaries~\ref{corodefix1} and~\ref{corodefix} for the torsion-free part of the factors in $H^1(\uppi_{g_1};\mathbb{Z}_{S_1})\otimes H^1(\uppi_{g_2};\mathbb{Z}_{S_2})$. Cocycle representatives for a set of $\mathbb{Z}_2$-generators of $T^{g_1,g_2}_{S_1\otimes S_2}$ are given by the $2(g_1+g_2)$ cocyles $\upalpha'_i$, $\upbeta'_i$ $(1\le i\le g_1)$ and $\upalpha''_j$, $\upbeta''_j$ $(1\le j\le g_2)$ defined by
\begin{align*}
\uplambda'_i=&\sum_{m_j\in S_2}(\uplambda_i\otimes\upmu_j)^*+\epsilon_{\uplambda_i}\left\{\begin{matrix}
(\upomega\otimes\upchi)^*,&\mbox{if \ $\widehat{\,\ell_i\,}\in S_1$}\\ 0,&\mbox{if \ $\widehat{\,\ell_i\,}\not\in S_1$} 
\end{matrix}\right\},\\
\upmu''_j=&\sum_{\ell_i\in S_1}(\uplambda_i\otimes\upmu_j)^*-\epsilon_{\upmu_j}\left\{\begin{matrix}
(\upchi\otimes\upomega)^*,&\mbox{if \ $\widehat{m_j}\in S_2$}\\ 0,&\mbox{if \ $\widehat{m_j}\not\in S_2$} 
\end{matrix}\right\}.
\end{align*}
An actual $\mathbb{Z}_2$-basis of $T_{S_1\otimes S_2}^{g_1,g_2}$ is obtained by removing, from the above set of generators, any one of the cocycles $\uplambda'_i$ with $\ell_i\in S_1$. Alternatively, any one of the cocycles $\upmu''_j$ with $m_j\in S_2$ can be removed in order to get a $\mathbb{Z}_2$-basis.
\end{proposition}
\begin{proof}
The first two assertions follow from the K\"unneth formula and Corollaries~\ref{muyinmediato}--\ref{corodefix}. The third assertion follows from~(\ref{diffi2A}) and~(\ref{diffi2B}), for if a cocycle $z\in\Hom_{\uppi_{g_1,g_2}}(M_2^{g_1,g_2};\mathbb{Z}_{S_1\otimes S_2})$ represents a cohomology class whose double vanishes, then $2z$ is a linear combination of the doubles of the cocycles $\uplambda'_i$ and $\upmu''_j$, and the conclusion holds from the fact that $\Hom_{\uppi_{g_1,g_2}}(M_2^{g_1,g_2};\mathbb{Z}_{S_1\otimes S_2})$ is torsion-free. Lastly, the fourth and fifth assertions follow from the series of equalities
{\small\begin{align*}
\sum_{\ell_i\in S_1}\uplambda'_i&=\sum_{\ell_i\in S_1}\left(\,\sum_{m_j\in S_2}(\uplambda_i\otimes\upmu_j)^*+\epsilon_{\uplambda_i}\left\{\begin{matrix}(\upomega\otimes\upchi)^*,&\mbox{if \ $\widehat{\,\ell_i\,}\in S_1$}\\ 0,&\mbox{if \ $\widehat{\,\ell_i\,}\not\in S_1$}\end{matrix}\right\}\right)\\
&=\sum\limits_{\underset{\mbox{\scriptsize $m_j\in S_2$}}{\mbox{\scriptsize $\ell_i\in S_1$}}}\!\!\! (\uplambda_i\otimes\upmu_j)^*+\!\!\!\sum\limits_{\underset{\mbox{\scriptsize $\widehat{\,\ell_i\,}\in S_1$}}{\mbox{\scriptsize $\ell_i\in S_1$}}} \!\!\!\epsilon_{\uplambda_i}(\upomega\otimes\upchi)^*=\!\!\!\sum\limits_{\underset{\mbox{\scriptsize $m_j\in S_2$}}{\mbox{\scriptsize $\ell_i\in S_1$}}} \!\!\!(\uplambda_i\otimes\upmu_j)^*
=\!\!\!\sum\limits_{\underset{\mbox{\scriptsize $\ell_i\in S_1$}}{\mbox{\scriptsize $m_j\in S_2$}}} \!\!\!(\uplambda_i\otimes\upmu_j)^*-\!\!\!\sum\limits_{\underset{\mbox{\scriptsize $\widehat{m_j}\in S_2$}}{\mbox{\scriptsize $m_j\in S_2$}}} \!\!\!\epsilon_{\upmu_j}(\upchi\otimes\upomega)^*\\
&=\sum_{m_j\in S_2}\left(\,\sum_{\ell_i\in S_1}(\uplambda_i\otimes\upmu_j)^*-\epsilon_{\upmu_j}\left\{\begin{matrix}(\upchi\otimes\upomega)^*,&\mbox{if \ $\widehat{m_j}\in S_2$}\\ 0,&\mbox{if \ $\widehat{m_j}\not\in S_2$}\end{matrix}\right\}\right)=\sum_{m_j\in S_2}\upmu''_j,
\end{align*}
\normalsize where the third and fourth equalities hold because $\epsilon_{\upalpha_k}+\epsilon_{\upbeta_k}=0$ for $k=1,\ldots,g_i$ and $i=1,2$.}
\end{proof}

\begin{proposition}\label{cdelpdim3}
For $S_1\neq\varnothing\neq S_2$, choose generators $\ell_0\in\{a_1(g_1),b_1(g_1),\ldots,a_{g_1}(g_1),b_{g_1}(g_1)\}$  and $m_0\in\{a_1(g_2),b_1(g_2),\ldots,a_{g_2}(g_2),b_{g_2}(g_2)\}$ with $\widehat{\,\ell_0\,}\in S_1$ and $\widehat{m_0}\in S_2$. Let $\uplambda_0$ and $\upmu_0$ stand for the corresponding greek-letter elements. Then $H^3(\uppi_{g_1,g_2};\mathbb{Z}_{S_1\otimes S_2})$ is a $\mathbb{Z}_2$-vector space of dimension $2(g_1+g_2)-1$ generated by the (cohomology classes of the) cocycles $\overline{\upalpha}'_i$, $\overline{\upbeta}'_i$ $(1\leq i \leq g_1)$ and $\overline{\upalpha}''_j$, $\overline{\upbeta}''_j$ $(1\leq j \leq g_2)$ defined by
\begin{align*}
\overline{\uplambda}_i'&=\epsilon_{\uplambda_i}(\uplambda_i\otimes\upomega)^*-\left\{\begin{matrix} \epsilon_{\upmu_0}(\upomega\otimes\upmu_0)^*, & \mbox{if $\widehat{\,\ell_i\,}\in S_1$} \\ 0, & \mbox{if $\widehat{\,\ell_i\,}\not\in S_1$}
 \end{matrix} \right\}, \\
\overline{\upmu}''_j&=\left\{\begin{matrix}\epsilon_{\uplambda_0}(\uplambda_0\otimes\upomega)^*, & \widehat{m_j}\in S_2 \\  0, & \widehat{m_j}\not\in S_2 \end{matrix}\right\}-\epsilon_{\upmu_j}(\upomega\otimes\upmu_j)^*.
\end{align*}
The above set of generators is in fact a $\mathbb{Z}_2$-basis (note that $\epsilon_{\uplambda_0}(\uplambda_0\otimes\upomega)^*-\epsilon_{\upmu_0}(\upomega\otimes\upmu_0)^*$ is counted twice, as it agrees with $\overline{\uplambda}'_0$ and $\overline{\upmu}''_0)$.
\end{proposition}
\begin{proof}
The group structure of $H^3(\uppi_{g_1,g_2};\mathbb{Z}_{S_1\otimes S_2})$ follows from the K\"unneth formula. The fact that the elements $\overline{\uplambda}'_i$ and $\overline{\upmu}''_j$ are cocycles generating the kernel of $d_4^*$ is transparent from~(\ref{diffi4A}) and~(\ref{diffi4B}). The result follows by noticing that there are $2(g_1+g_2)-1$ such cocycles, as $\epsilon_{\uplambda_0}(\uplambda_0\otimes\upomega)^*-\epsilon_{\upmu_0}(\upomega\otimes\upmu_0)^*$ is counted twice: $\overline{\uplambda}'_0=\overline{\upmu}''_0$.
\end{proof}

\begin{remark}\label{quepasaendimension4}
By the K\"unneth formula, $H^4(\uppi_{g_1,g_2};\mathbb{Z}_{S_1\otimes S_2})=H^2(\uppi_{g_1};\mathbb{Z}_{S_1})\otimes H^2(\uppi_{g_2};\mathbb{Z}_{S_2})=\mathbb{Z}_2$ provided either $S_1\neq\varnothing$ or $S_2\neq\varnothing$, with generator given by the cohomology class of the cocycle $\uptau:=(\upomega\otimes\upomega)^*$, i.e., the exterior product of the generators of the factors in the tensor product.
\end{remark}

Using Lemmas~\ref{cochain-level cup-product} and~\ref{signconventionfinal}, Propositions~\ref{cdelpdim1}--\ref{cdelpdim3}, and Remark~\ref{quepasaendimension4}, it is straightforward to read off the structure of any given cup-product $H^*(\uppi_{g_1,g_2};\mathbb{Z}_{S_1\otimes S_2})\otimes H^*(\uppi_{g_1,g_2};\mathbb{Z}_{S'_1\otimes S'_2})\to H^*(\uppi_{g_1,g_2};\mathbb{Z}_{(S_1\ominus S'_1)\otimes(S_2\ominus S'_2)})$.

\section{Effective topological complexity}\label{robotics}
Recall that $\Sigma_g$ (respectively, $\uppi_g$) stands for the closed orientable surface of genus $g$ (respectively, for the fundamental group of $\Sigma_g$). Embed $\Sigma_g$ in $\mathbb{R}^3$ as in Figure~\ref{embedded}, so that $\Sigma_g$ is invariant under reflections in the $xy$-, $xz$- and $yz$-planes. The orientation-reversing involution $\sigma\colon\Sigma_g\to\Sigma_g$ given by $\sigma(x,y,z)=(-x,-y,-z)$ yields an ``antipodal'' action of the group $\z_2=\{e, \sigma\}$ on $\Sigma_g$. The orbit space $N_{g+1}:=\Sigma_g/\z_2$ is a non-orientable surface of genus $g+1$.
\begin{figure}[h!]
\begin{center}
\begin{tikzpicture}[scale=0.65]
\filldraw[black] (20.62,0) circle (1.5pt) node[right] {$\sigma\cdot\star$};
\filldraw[black] (-1.5,0) circle (1.5pt) node[left] {$\star$};
\node at (7.5,0,0){$\dotsb$};
\node at (11.5,0,0){$\dotsb$};
\node[below] at (19.21,-0.7){$H(\star,-)$};
\draw[thick,->, dashed] (9,0,0) -- (9.9,0,0) node[right]{$x$};
\draw[thick,->, dashed] (9,0,0) -- (9.8,1,4) node[below]{$y$};
\draw[thick,->, dashed] (9,0,0) -- (9.8,2,2.01) node[above]{$z$};
\draw (0.8,-0.01) arc (0:180:0.8 and 0.35); 
\draw (-0.85,0.1) arc (180:360:0.85 and 0.4); 
\draw (4.8,-0.01) arc (0:180:0.8 and 0.35); 
\draw (3.15,0.1) arc (180:360:0.85 and 0.4); 
\draw (15.8,-0.01) arc (0:180:0.8 and 0.35); 
\draw (14.15,0.1) arc (180:360:0.85 and 0.4); 
\draw (19.8,-0.01) arc (0:180:0.8 and 0.35); 
\draw (18.15,0.1) arc (180:360:0.85 and 0.4); 
\draw (1.0605,0.5656) arc (45:315:1.5 and 0.8); 
\draw (1.0605,0.5656) to[out=-28.1,in=200] (3,.6); 
\draw (5.12,0.60) arc (45:135:1.5 and 0.8); 
\draw (5.12,0.6) to[out=-28.1,in=200] (7,.6);
\draw (1.0605,-0.5656) to[out=28.1,in=160] (3,-0.5756); 
\draw (2.998,-0.576) arc (225:315:1.5 and 0.8); 
\draw (7.23,-0.578) arc (45:135:1.5 and 0.8); 
\draw (18.0605,-0.5656) arc (-135:135:1.5 and 0.8);
\draw (18.06,-0.566) arc (45:135:1.5 and 0.8); 
\draw (13.82,-0.567) arc (225:315:1.5 and 0.8); 
\draw (13.82,-0.566) arc (45:135:1.5 and 0.8); 
\draw (15.94,0.565) arc (225:315:1.5 and 0.8); 
\draw (15.94,0.566) arc (45:135:1.5 and 0.8); 
\draw (11.70,0.566) arc (225:315:1.5 and 0.8);
\path[overlay,-, font=\scriptsize,>=latex, line width=0.2mm]
(-1.5,0) edge[out=315,in=200] node[] {} (1.5,-0.1)
(3,-0.3) edge[out=340,in=200] node[] {} (5.5,-0.2)
(13.4,-0.15) edge[out=340,in=215] node[] {} (16.3,-0.2)
(17.8,-0.3) edge[out=340,in=215] node[] {} (20.62,0);
\path[overlay,-, font=\scriptsize,>=latex, line width=0.2mm]
(1.5,-0.1) edge[out=20,in=160] node[] {} (3,-0.3)
(5.5,-0.2) edge[out=20,in=160] node[] {} (7.3,-0.4)
(11.6,-0.4) edge[out=20,in=160] node[] {} (13.4,-0.15)
(16.3,-0.2) edge[out=20,in=160] node[] {} (17.8,-0.3);
\draw[->, >=latex,] (7.3,-0.4) to (7.33,-0.41);
\draw[->, >=latex,] (19.21,-0.53) to (19.25,-0.53);
\end{tikzpicture}
\end{center}
\caption{The embedding $\Sigma_g\subset\mathbb{R}^3$. The path $H(\star,-)$ will be relevant later}
\label{embedded}
\end{figure}
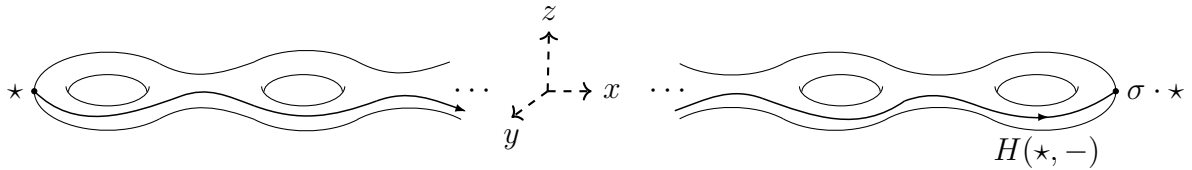

For $i \in \{0,1 \}$ let $e_i:P\Sigma_g \to \Sigma_g$ be the fibration defined by $e_i(\gamma) = \gamma(i)$, where $P\Sigma_g$ stands for the free-path space of $\Sigma_g$. Consider the pullback of the maps $\epsilon_0:P\Sigma_g \xrightarrow{e_0} \Sigma_g \xrightarrow{q} N_{g+1}$ and $\epsilon_1:P\Sigma_g \xrightarrow{e_1} \Sigma_g \xrightarrow{q} N_{g+1}$, where $q$ stands for the canonical quotient projection. Explicitly, $P_2\Sigma_g$ is the subspace of the cartesian product $P\Sigma_g\times P\Sigma_g$ consisting of the pairs $(\upalpha,\upbeta)$ of paths in $\Sigma_g$ such that the starting point of $\upbeta$ lies in the $\mathbb{Z}_2$-orbit of the ending point of $\upalpha$.

\smallskip
The next construction, a particular instance of B{\l}aszczyk-Kaluba's definition of effective topological complexity of a space with a group action~(\cite{MR3738182}), arises from the idea of taking advantage of the symmetries present in the space of states of a given autonomous robot in order to efficiently motion plan the robot's tasks.

\begin{definition}\label{def_Teff}
Let $p_2: P_2\Sigma_g\to \Sigma_g \times \Sigma_g$ be the fibration given by $p_2(\upalpha, \upbeta)=(\upalpha(0), \upbeta(1))$. The (reduced) effective topological complexity of $\Sigma_g$ with respect to its antipodal action, denoted by $\TC^{\z_2}(\Sigma_g)$, is the sectional category $\g(p_2)$. Explicitly, $\TC^{\z_2}(\Sigma_g)$ is the smallest non-negative integer $k$ such that $\Sigma_g \times \Sigma_g$ can be covered by open sets $U_0, \ldots, U_{k}$ in such a way that, for every $i=0, \ldots, k$, there exists a map $s_i:U_i \to P_2\Sigma_g$ such that $p_2 \circ s_i$ is the inclusion $U_i \hookrightarrow \Sigma_g \times \Sigma_g$.
\end{definition}

The proof of Theorem~\ref{etcso} is based on the following observations.

The antipodal involution $\sigma$ on $\Sigma_g$ has no fixed points, so that $P_2\Sigma_g$ can be identified with $\z_2\times P\Sigma_g$, the topological disjoint union of two copies of $P\Sigma_g$. The ``trivial'' copy $\{e\}\times P\Sigma_g$ accounts for the ``honest'' paths, i.e., pairs $(\upalpha,\upbeta)$ in $P_2\Sigma_2$ satisfying $\upalpha(1)=\upbeta(0)$, while the ``shifted'' copy $\{\sigma\}\times P\Sigma_g$ accounts for the ``broken'' paths, i.e., pairs $(\upalpha,\upbeta)\in P_2\Sigma_2$ satisfying $\upalpha(1)=\sigma\cdot\upbeta(0)$. In these terms, $p_2$ takes the form $q_2\colon \z_2\times P\Sigma_g \to \Sigma_g\times\Sigma_g$ with $q_2(\sigma^i,\upgamma)=(\upgamma(0),\sigma^i\cdot\upgamma(1))$, for $i=0,1$. In other words, $q_2$ is the double evaluation map $e_{0,1}\colon P\Sigma_g\to\Sigma_g\times\Sigma_g$ ---whose sectional category defines Farber's $\TC(\Sigma_g)$--- on the honest paths while, on the broken maps, $q_2$ becomes the ``twisted'' double evaluation map $e_{0,1}'\colon P\Sigma_g\to\Sigma_g\times\Sigma_g\hspace{.3mm}$ given by $e_{0,1}'(\gamma)=(\gamma(0),\sigma\cdot\gamma(1))$.

For a point $x\in\Sigma_g$, let $c_x\in P\Sigma_g$ denote the constant path at $x$. This yields a standard homotopy equivalence $\Sigma_g\simeq P\Sigma_g$. The ``saturated diagonal' $j\colon\z_2\times\Sigma_g\hookrightarrow\Sigma_g\times\Sigma_g$ is the composite $\z_2\times\Sigma_g\simeq \z_2\times P\Sigma_g\stackrel{q_2}\longrightarrow\Sigma_g\times \Sigma_g$, i.e., $j(\sigma^i,x)=(x,\sigma^i\cdot x)$, which has two co-components: the standard diagonal map $\Delta=(1,1)\colon\Sigma_g\to\Sigma_g\times\Sigma_g$ and the ``twisted'' diagonal map $\Delta'=(1,\sigma)\colon\Sigma_g\to\Sigma_g\times\Sigma_g$.

Theorem~\ref{etcso} follows at once from Propositions~\ref{auxi} and~\ref{ixua} below. 
\begin{proposition}[{\cite[Theorem~4, page 73]{Sch58}}]\label{auxi}
The sectional category of a fibration $p\colon E\to B$ is bounded from below by the product-length of elements in the kernel of the map induced in (twisted) cohomology by $p$. Explicitly, assume that $\mathcal{A}_i$ ($1\leq i\leq\ell$) are local coefficient systems over $B$ and that cohomology classes $x_i\in H^*(B;\mathcal{A}_i)$ are given with $p^*(x_i)=0$ for $i=1,\ldots,\ell$. If $0\neq x_1\cdots x_\ell\in H^*(B,\bigotimes_i \mathcal{A}_i)$, then $\g(p)\geq\ell$.
\end{proposition}
\begin{definition}\label{e-z-d}
Let $\mathcal{A}$ be a local coefficient system on $\Sigma_g \times \Sigma_g$. A class $x\in H^* (\Sigma_g \times \Sigma_g; \mathcal{A})$ in the kernel of the map induced by $j$ is called an effective zero-divisor in $\Sigma_g$.
\end{definition}
\begin{proposition}\label{ixua}
For $g\geq 2$, the elements $a,b\in H^1(\uppi_g\times\uppi_g;\mathbb{Z}_\varnothing)$ and $c\in H^2(\uppi_g\times\uppi_g;\mathbb{Z}_\varnothing)$ given by $a=(\upalpha_1\otimes\upchi)^*-(\upchi\otimes\upalpha_1)^*+(\upchi\otimes\upalpha_g)^*-(\upalpha_g\otimes\upchi)^*$, $b=(\upbeta_1\otimes\upchi)^*-(\upchi\otimes\upbeta_1)^*+(\upbeta_g\otimes\upchi)^*-(\upchi\otimes\upbeta_g)^*$, and $c=(\upbeta_1\otimes\upalpha_1)^*-(\upbeta_g\otimes\upalpha_g)^*$
are effective zero-divisors in $\Sigma_g$ having non-zero cup-product.
\end{proposition}

We neglect any difference between cohomology classes and their cochain representatives in Proposition~\ref{ixua}, as coboundaries vanish with trivial coefficients (see for instance~(\ref{ytrcnvmbn})). The rest of the section is devoted to the proof of Proposition~\ref{ixua}.

By definition, effective zero-divisors in $\Sigma_g$ are elements in the intersection of the kernels of the maps induced in cohomology by the two diagonals $\Delta,\Delta'\colon\Sigma_g\to\Sigma_g\times\Sigma_g$. Since all spaces insight are aspherical, cohomology calculations can (and will) be performed at the level of fundamental groups. Care is needed in such an approach, for $\Delta$ and~$\Delta'$ cannot be considered simultaneously as based maps.

Pick once and for all a homotopy $H\colon\Sigma_g\times [0,1]\to\Sigma_g$ so that $H(-,0)$ is the identity, while $H(-,1)$ takes $\star$ into $\sigma\cdot\star$, where $\star$ is the base point indicated in Figure~\ref{embedded}. The existence of such a homotopy follows from the fact that $\Sigma_g$ is a manifold. In fact, $H$ can be choosen so that $H(\star,-)$ is the path shown in Figure~\ref{embedded}. Then $G=(\text{proj},\sigma\circ H)\colon \Sigma_g\times[0,1]\to\Sigma_g\times\Sigma_g$ is a homotopy starting at $G(-,0)=(1,\sigma)=\Delta'$, while the ending branch $G(-,1)$ is a based map taking the form
\begin{equation}\label{endingbranch}
\Delta'':=(1,\mu)=(1\times\mu)\circ\Delta\colon(\Sigma_g,\star)\to\left(\rule{0mm}{4mm}\Sigma_g\times\Sigma_g,(\star,\star)\right),
\end{equation}
where $\mu=\sigma\circ H(-,1)$. It is standard (see~\cite[Chapter~VI, Section~2]{WhHot}) that, for any system of coefficients $\mathcal{A}$ on $\Sigma_g\times\Sigma_g$, the homotopy $G$ determines an isomorphism between the induced systems of coefficients $(\Delta')^*\mathcal{A}$ and $(\Delta'')^*\mathcal{A}$ on $\Sigma_g$ in such a way that the resulting cohomology isomorphism fits in a commutative diagram
$$\xymatrix{
H^*(\Sigma_g\times\Sigma_g;\mathcal{A})\ar[r]^(.505){(\Delta')^*}\ar[dr]_(.48){(\Delta'')^*} & 
H^*(\Sigma_g;(\Delta')^*\mathcal{A}) \ar[d]^\cong \\ 
 & H^*(\Sigma_g;(\Delta'')^*\mathcal{A}).} 
$$

The point is that $\ker(\Delta')^*=\ker(\Delta'')^*$. Since $\Delta$ and $\Delta''$ are simultaneously based maps (both sending $\star$ into $(\star,\star)$), effective zero-divisors, i.e.~$\ker(\Delta)^*\cap\ker(\Delta'')^*$, can then be computed, in purely algebraic terms, from a description of the induced maps $\Delta_*$ and $\Delta''_*$ at the level of fundamental groups. Note that $\Delta_*$ is the diagonal inclusion $\uppi_g\hookrightarrow\uppi_g\times\uppi_g$, so  its cohomology effect follows directly from the results in previous sections of the paper. On the other hand,~(\ref{endingbranch}) shows that the case of $\Delta''_*$ amounts to a suitable understanding of the map induced by $\mu$ at the level of fundamental groups.

\begin{definition}
A word $w=\ell_1\ell_2\cdots\ell_{2t+1}$ $(t\geq0)$ on letters $\ell_j\in\{a_i,b_i,\overline{a_i},\overline{b_i}\colon 1\leq i\leq g\}$ is said to be almost-paired if there is a cardinality-$t$ subset $P\subset\{1,2,\ldots,2t+1\}$ and an injective map $p\colon P\to\{1,2,\ldots,2t+1\}-P$ such that $\ell_{p(j)}=\overline{\ell_j}$, for each $j\in P$. In these conditions, there is a single unpaired integer, i.e., a single integer $j_0\in\{1,2,\ldots,2t+1\}$ that does not lie in $P$ nor in the image of $p$. By abuse of notation, the existence of $P$ and $p$ as above will simply be indicated by writing $w=\mathcal{P}(\ell_{j_0})$. Further, the latter conventional equality will be applied to elements $w\in\uppi_g$, in the understanding that the equality $w=\ell_1\ell_2\cdots\ell_{2t+1}$ holds as elements of $\uppi_g$.
\end{definition}

\begin{lemma}\label{delosdia}
Generators $a_i$ and $b_i$ $(1\leq i\leq g)$ of $\uppi_g$ as in~(\ref{presentacion}) can be chosen so that the map $\mu_*\colon\uppi_g\to\uppi_g$ induced in fundamental groups by $\mu\colon(\Sigma_g,\star)\to(\Sigma_g,\star)$ satisfies
$\mu_*(a_i)=\mathcal{P}(\overline{a_{g-i+1}})$ \ and \ $\mu_*(b_i)=\mathcal{P}(b_{g-i+1}),$ \ for all $i\in\{1,2,\ldots,g\}$.
\end{lemma}

We deal with Lemma~\ref{delosdia} after deducing Proposition~\ref{ixua} from the lemma.

\begin{proof}[Proof of Proposition~\ref{ixua}]
Routine calculation using Example~\ref{hhaayyttsmmmldksju} and Lemmas~\ref{cochain-level cup-product} and~\ref{signconventionfinal} gives that, as a cocycle, the product $abc$ agrees with $2(\upomega\otimes\upomega)^*$ ---twice a representative for a generator of $H^4(\uppi_g\times\uppi_g;\mathbb{Z}_\varnothing)\cong\mathbb{Z}$. It is also elementary to check that $a$, $b$ and $c$ are usual zero-divisors, i.e, that they lie in the kernel of the map induced by the (usual) diagonal $\Delta$. We check below that $a$, $b$ and $c$ also pull-back trivially under the map $\Delta''$ in~(\ref{endingbranch}). 

At the level of resolutions, the map induced by $1\times\mu$ can be computed as a tensor product. So we focus on the $\mu$-component, for which we will only need to describe the first two stages $\mu_0$ and $\mu_1$ of the chain map
$$\xymatrix{
0 & \mathbb{Z} \ar[l] \ar[d]^{=} & M_0 \ar[l] \ar[d]^{\mu_0} & M_1 \ar[l]_{ \ d_1} \ar[d]^{\mu_1} & M_2 \ar[l] \ar[d]^{\mu_2} & 0 \ar[l] \\
0 & \mathbb{Z} \ar[l] & M_0 \ar[l] & M_1 \ar[l] & M_2 \ar[l] & 0 \ar[l]
}$$
induced by $\mu_*\colon\uppi_g\to\uppi_g$ and, more importantly, describe the map $\mu_1^*$ at the cochain level (in~(\ref{kuppa}) below).

\smallskip
It is obvious that $\mu_0$ is determined by $\mu_0(\upchi)=\upchi$ whereas, as in Proposition~\ref{alalyvyvnene}, $\mu_1(\lambda)$ can be taken as $s_0\circ\mu_0\circ d_1(\lambda)$ for basis elements $\lambda\in M_1$. Following the convention set up in the paragraph following~(\ref{cochainproduct}), we have
$$s_0(\mu_0(d_1(\lambda)))=s_0(\mu_0((\ell-1)\cdot\upchi))=s_0(\mu_*(\ell)\cdot \upchi-\upchi)=s_0(\mu_*(\ell)\cdot \upchi).$$
The evaluation of the latter expression requires, in principle, knowledge of the normal form of $\mu_*(\ell)$. We explain next how such a requirement can be waived. Given the form of the rewriting rules (\ref{rrsimples})--(\ref{R8}), it is clear that the normal form of an element $w=\mathcal{P}(\ell)\in\uppi_g$ has again the form $N(w)=\mathcal{P}(\ell)$ ---the latter equality uses the element $\ell$ in the former equality. Furthermore, from its definition (in~(\ref{deritotalFox}) and~(\ref{dtf})), the $s_0$-image of an element $\mathcal{P}(\ell)\cdot\upchi$ (with $\mathcal{P}(\ell)$ written in normal form) is a sum of terms
\begin{equation}\label{parte1}
(x-y)\cdot\upeta,
\end{equation}
with $x,y\in\uppi_g$ and $\upeta$ a basis element, plus an additional summand $z\cdot s_0(\ell\cdot\upchi)$, with $z\in\uppi_g$. The point is that terms~(\ref{parte1}) can all be neglected when applying a functor Hom$_{\pi_g}(-,\mathbb{Z}_\varnothing)$. Together with Lemma~\ref{delosdia}, this yields that, at the cochain level,
\begin{equation}\label{kuppa}
\mu_1^*(\upalpha^*_i)=-\upalpha^*_{g-i+1}\quad\mbox{and\;}\quad \mu_1^*(\upbeta^*_i)=\upbeta^*_{g-i+1}
\end{equation}
for $1\leq i\leq g$. The conclusion follows now from direct calculation:
\begin{align*}
(\Delta'')^*(a)&=\Delta^*\circ (1\otimes\mu^*)\left((\upalpha_1\otimes\upchi)^*-(\upchi\otimes\upalpha_1)^*+(\upchi\otimes\upalpha_g)^*-(\upalpha_g\otimes\upchi)^*\rule{0mm}{4mm}\right)\\
&=\Delta^*\circ (1\otimes\mu^*)\left(\upalpha^*_1\otimes\upchi^*-\upchi^*\otimes\upalpha_1^*+\upchi^*\otimes\upalpha_g^*-\upalpha^*_g\otimes\upchi^*\rule{0mm}{4mm}\right)\\
&=\Delta^*\left(\upalpha^*_1\otimes\upchi^*+\upchi^*\otimes\upalpha_g^*-\upchi^*\otimes\upalpha_1^*-\upalpha^*_g\otimes\upchi^*\rule{0mm}{4mm}\right)=\upalpha^*_1+\upalpha^*_g-\upalpha^*_1-\upalpha^*_g=0.
\end{align*}
Note that the standard sign convention introduces no sign in the second and third equalities above, as the relevant morphisms have degree zero. The analysis of $(\Delta'')^*(b)$ is similar, whereas
$
(\Delta'')^*(c)=\Delta^*\circ (1\otimes\mu^*)\left(-\upbeta^*_1\otimes\upalpha^*_1+\upbeta^*_g\otimes\upalpha^*_g\rule{0mm}{4mm}\right)=\Delta^*\left(\upbeta^*_1\otimes\upalpha^*_g-\upbeta^*_g\otimes\upalpha_1^*\rule{0mm}{4mm}\right)
$, with both $\Delta^*(\upbeta^*_1\otimes\upalpha^*_g)$ and $\Delta^*(\upbeta^*_g\otimes\upalpha_1^*)$ being zero.
\end{proof}

The proof of Lemma~\ref{delosdia} is a straightforward geometric exercise based on the fact that, since $\sigma\circ H$ is a homotopy from $\sigma$ to $\mu$, \cite[Lemma~1.19]{MR1867354} yields a commutative diagram
$$\xymatrix{
\uppi_g\ar[r]^(.505){\mu_*}\ar[dr]_(.48){\sigma_*} & 
\uppi_g \\ 
 & \uppi_1(\Sigma_g,\sigma\cdot\star) \ar[u]_{\cong}.}
$$
The point is that the vertical isomorphism (conjugation by the path $\sigma\circ H(\star,-)$) becomes geometrically explicit once the path $H(\star,-)$ is chosen (as we did in Figure~\ref{embedded}). Likewise, the morphism $\sigma_*$ is geometrically explicit in terms of the representing loops described in Figure~\ref{la2} for the generators in~(\ref{presentacion}).
\begin{figure}[h!]    \centering
\begin{tikzpicture}
\draw [-, rounded corners, thick] (0,0) -- (4,0) -- (4,3) -- (-4,3) -- (-4,0) --(0,0);
\draw (2.5,1.5) circle (0.5cm);
\draw (0,1.5) circle (0.5cm);
\draw (-2.5,1.5) circle (0.5cm);
\filldraw[black] (-4,1.5) circle (2pt) node[left] {$\star$};
\node at (-2.5,1.5){$1$};
\node at (0,1.5){$i$};
\node at (2.5,1.5){$g$};
\node at (-1.25,1.5){$\cdots$};
\node at (1.25,1.5){$\cdots$};
\draw[line width=0.2mm, rotate=0, dashed] (0,2) arc (-90:90:0.2 and 0.5); 
\draw[line width=0.4mm, rotate=0] (0,3) arc (90:170:0.2 and 0.5);
\draw[line width=0.4mm, rotate=0] (-0.2,2.415) arc (190:271:0.2 and 0.5);
\draw[->, >=latex] (-0.16, 2.8) to (-0.12, 2.9);
\path[overlay,-, font=\scriptsize,>=latex, line width=0.4mm]
(-1.3, 0.8) edge[out=0,in=180] node[] {} (-0.1965, 2.6);
\path[overlay,-, font=\scriptsize,>=latex, line width=0.4mm]
(-1.2, 0.65) edge[out=0,in=180] node[] {} (-0.1966, 2.41);
\path[overlay,-, font=\scriptsize,>=latex, line width=0.4mm]
(-3.95,1.51) edge[out=320,in=180] node[] {} (-1.3, 0.8);
\path[overlay,-, font=\scriptsize,>=latex, line width=0.4mm]
(-3.95,1.48) edge[out=315,in=180] node[] {} (-1.2, 0.65);
\node at (0.5,2.5){$a_i$};
\draw[line width=0.4mm, rotate=0] (-0.39,1) arc (-136:98:0.6 and 0.65); 
\draw[line width=0.4mm, rotate=0] (-0.2,2.07) arc (100:220:0.5 and 0.6); 
\draw[->, >=latex] (0.64, 1.4) to (0.64, 1.5);
\node at (0.7,.8){$b_i$};
\path[overlay,-, font=\scriptsize,>=latex, line width=0.4mm]
(-3.93,1.47) edge[out=280,in=-97] node[] {} (-0.5, 1.1);
\path[overlay,-, font=\scriptsize,>=latex, line width=0.4mm]
(-3.95,1.47) edge[out=270,in=-100] node[] {} (-0.39, 1);
\end{tikzpicture}
    \caption{Loops representing generators in~(\ref{presentacion})}
    \label{la2}
\end{figure}
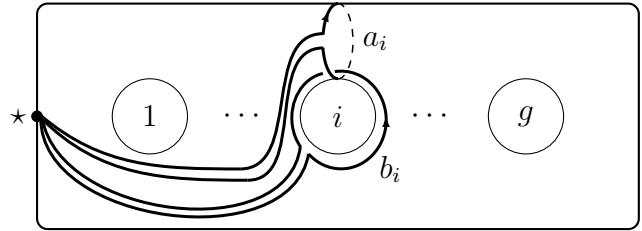
In fact, with $P_j:=[a_1,b_1]\cdots[a_j,b_j]\in\uppi_g$ ($0\leq j\leq g$), representing loops for the elements $P_ia_{i+1}, \,P_ia_{i+1}b_{i+1}, \,P_ia_{i+1}b_{i+1}\overline{a_{i+1}}, \,P_{i+1} \in\uppi_g$ ($0\leq i<g$) are given in Figures~\ref{elementsa}--\ref{elementsd}.
\begin{figure}[h!]
    \centering
    \begin{tikzpicture}[thick,scale=0.8, every node/.style={scale=0.8}]
\draw [-, rounded corners, thick] (0,0) -- (6,0) -- (6,3) -- (-6,3) -- (-6,0) --(0,0);
\draw (5,1.5) circle (0.5cm);
\draw (-1,1.5) circle (0.5cm);
\draw (1,1.5) circle (0.5cm);
\draw (-5,1.5) circle (0.5cm);
\node at (-5,1.5){\footnotesize$1$};
\node at (-1,1.5){\footnotesize$i$};
\node at (1,1.5){\scriptsize$i{+}1$};
\node at (5,1.5){\footnotesize$g$};
\node at (-3,1.5){$\cdots$};
\node at (3,1.5){$\cdots$};
\filldraw[black] (-6,1.5) circle (2pt) node[left] {$\star$};
\draw[line width=0.2mm, rotate=0, dashed] (1,0) arc (-90:90:0.2 and 0.5); 
\draw[line width=0.4mm, rotate=0] (1,1) arc (90:170:0.2 and 0.5);
\draw[line width=0.4mm, rotate=0] (0.8,0.415) arc (190:271:0.2 and 0.5);
\path[overlay,-, font=\scriptsize,>=latex, line width=0.4mm]
(-5.95,1.44) edge[out=320,in=180] node[] {} (0.8, 0.415);
\path[overlay,-, font=\scriptsize,>=latex, line width=0.4mm]
(-5.95,1.52) edge[out=320,in=180] node[] {} (0.8, 0.6);
\draw[->, >=latex] (0.83, .23) to (.87, .1);
\end{tikzpicture}
    \caption{$P_ia_{i+1}$}
    \label{elementsa}
\end{figure}
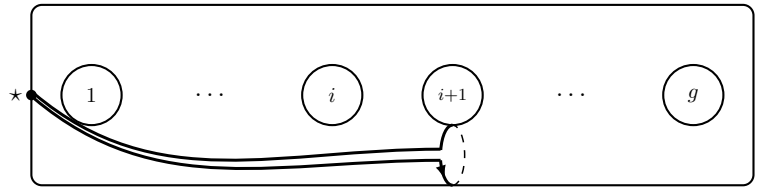
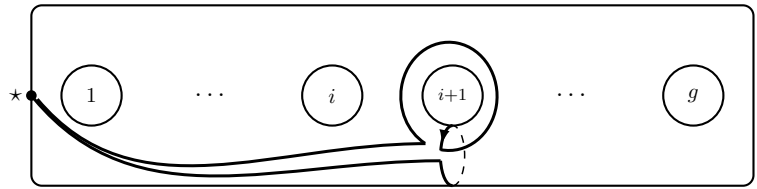
\begin{figure}[h!]
    \centering
\begin{tikzpicture}[thick,scale=0.8, every node/.style={scale=0.8}]
\draw [-, rounded corners, thick] (0,0) -- (6,0) -- (6,3) -- (-6,3) -- (-6,0) --(0,0);
\draw (5,1.5) circle (0.5cm);
\draw (-1,1.5) circle (0.5cm);
\draw (1,1.5) circle (0.5cm);
\draw (-5,1.5) circle (0.5cm);
\node at (-5,1.5){\footnotesize$1$};
\node at (-1,1.5){\footnotesize$i$};
\node at (1,1.5){\scriptsize$i{+}1$};
\node at (5,1.5){\footnotesize$g$};
\node at (-3,1.5){$\cdots$};
\node at (3,1.5){$\cdots$};
\filldraw[black] (-6,1.5) circle (2pt) node[left] {$\star$};
\draw[line width=0.2mm, rotate=0, dashed] (1,0) arc (-90:90:0.2 and 0.5); 
\draw[line width=0.4mm, rotate=0] (1,1) arc (90:170:0.2 and 0.5);
\draw[line width=0.4mm, rotate=0] (0.8,0.415) arc (190:271:0.2 and 0.5);
\path[overlay,-, font=\scriptsize,>=latex, line width=0.4mm]
(-5.95,1.44) edge[out=310,in=180] node[] {} (0.8, 0.415);
\draw[line width=0.4mm, rotate=0] (0.8,0.6) arc (-100:240:0.8 and 0.9);
\path[overlay,-, font=\scriptsize,>=latex, line width=0.4mm]
(-5.90,1.44) edge[out=310,in=180] node[] {} (0.55, 0.7);
\draw[->, >=latex] (0.84, 0.8) to (0.82, 0.7);
\end{tikzpicture}
    \caption{$P_ia_{i+1}b_{i+1}$}
    \label{elementsb}
\end{figure}
\begin{figure}[h!]
    \centering
\begin{tikzpicture}[thick,scale=0.8, every node/.style={scale=0.8}]
\draw [-, rounded corners, thick] (0,0) -- (6,0) -- (6,3) -- (-6,3) -- (-6,0) --(0,0);
\draw (5,1.5) circle (0.5cm);
\draw (-1,1.5) circle (0.5cm);
\draw (1,1.5) circle (0.5cm);
\draw (-5,1.5) circle (0.5cm);
\node at (-5,1.5){\footnotesize$1$};
\node at (-1,1.5){\footnotesize$i$};
\node at (1,1.5){\scriptsize$i{+}1$};
\node at (5,1.5){\footnotesize$g$};
\node at (-3,1.5){$\cdots$};
\node at (3,1.5){$\cdots$};
\filldraw[black] (-6,1.5) circle (2pt) node[left] {$\star$};
\draw[line width=0.2mm, rotate=0, dashed] (1.1,0) arc (-90:90:0.5 and 1.5); 
\draw[line width=0.4mm, rotate=0] (0.6,0.415) arc (190:271:0.4 and 0.5);
\draw[line width=0.4mm, rotate=0] (1.1,3) arc (90:210:0.7 and 1.5);
(-5.95,1.44) edge[out=310,in=180] node[] {} (0.8, 0.415);
\path[overlay,-, font=\scriptsize,>=latex, line width=0.4mm]
(-5.90,1.44) edge[out=310,in=180] node[] {} (0.5, 0.75);
\path[overlay,-, font=\scriptsize,>=latex, line width=0.4mm]
(-5.95,1.44) edge[out=310,in=180] node[] {} (0.6, 0.415);
\draw[->, >=latex] (0.415, 1.8) to (0.405, 1.7);
\end{tikzpicture}
    \caption{$P_ia_{i+1}b_{i+1}\overline{a_{i+1}}$}
    \label{elementsc}
\end{figure}
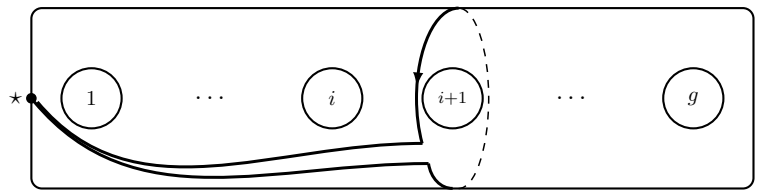
\begin{figure}[h!]
    \centering
\begin{tikzpicture}[thick,scale=0.8, every node/.style={scale=0.8}]
\draw [-, rounded corners, thick] (0,0) -- (6,0) -- (6,3) -- (-6,3) -- (-6,0) --(0,0);
\draw (5,1.5) circle (0.5cm);
\draw (-1,1.5) circle (0.5cm);
\draw (1,1.5) circle (0.5cm);
\draw (-5,1.5) circle (0.5cm);
\node at (-5,1.5){\footnotesize$1$};
\node at (-1,1.5){\footnotesize$i$};
\node at (1,1.5){\scriptsize$i{+}1$};
\node at (5,1.5){\footnotesize$g$};
\filldraw[black] (-6,1.5) circle (2pt) node[left] {$\star$};
\draw[line width=0.2mm, rotate=0, dashed] (2,0) arc (-90:90:0.25 and 1.5); 
\draw[line width=0.4mm, rotate=0] (1.8,0.415) arc (190:271:0.2 and 0.5);
\draw[line width=0.4mm, rotate=0] (2.1,3) arc (90:210:0.4 and 1.5);
(-5.95,1.44) edge[out=310,in=180] node[] {} (0.8, 0.415);
\path[overlay,-, font=\scriptsize,>=latex, line width=0.4mm]
(-5.95,1.44) edge[out=312,in=180] node[] {} (1.77, 0.75);
\path[overlay,-, font=\scriptsize,>=latex, line width=0.4mm]
(-5.95,1.44) edge[out=310,in=175] node[] {} (1.8, 0.415);
\draw[->, >=latex] (1.703, 1.7) to (1.703, 1.6);
\node at (-3,1.5){$\cdots$};
\node at (3.35,1.5){$\cdots$};
\end{tikzpicture}
    \caption{$P_{i+1}$}
    \label{elementsd}
\end{figure}

Note that $P_g$ is in particular the neutral element in $\uppi_g$, so the loops described in Figure~\ref{la2} do represent generators as required in~(\ref{presentacion}), and the algebraic results in previous sections can be applied in terms of these generators. Lemma~\ref{delosdia} is now verified by geometric inspection. For instance, Figures~\ref{sequence1}--\ref{sequence5} display the evolution of the elements $a_i,\, \rho_{yz}(a_i),\, \rho_{xz}\rho_{yz}(a_i),\, \sigma_*(a_i)=\rho_{xy}\rho_{xz}\rho_{yz}(a_i),$ and $\mu_*(a_i)$, where $\rho$-maps stand for reflections in the indicated hyperplanes. The formula $\mu_*(a_i)=\mathcal{P}(\overline{a_{g-i+1}})$ follows by observing that the last stage in Figure~\ref{sequence5} is conjugate to the inverse of $P_{g-i}a_{g-i+1}$.
The formula $\mu_*(b_i)=\mathcal{P}(b_{g-i+1})$ is verified in a similar fashion.

\begin{figure}[h!]
    \centering
\begin{tikzpicture}[thick,scale=0.8, every node/.style={scale=0.8}]
\draw [-, rounded corners, thick] (0,0) -- (6,0) -- (6,3) -- (-6,3) -- (-6,0) --(0,0);
\draw (5,1.5) circle (0.5cm);
\draw (-1,1.5) circle (0.5cm);
\draw (1,1.5) circle (0.5cm);
\draw (-5,1.5) circle (0.5cm);
\node at (-5,1.5){\footnotesize$1$};
\node at (-1,1.5){\footnotesize$i$};
\node at (1,1.5){\scriptsize{$g{-}i{+}1$}};
\node at (5,1.5){\footnotesize$g$};
\node at (-3,1.5){$\cdots$};
\node at (3,1.5){$\cdots$};
\node at (7,1.5){\rule{0mm}{4mm}};
\filldraw[black] (-6,1.5) circle (2pt) node[left] {$\star$};
\draw[ line width=0.2mm, rotate=0, dashed] (-1,2) arc (-90:90:0.2 and 0.5); 
\draw[ line width=0.4mm, rotate=0] (-1,3) arc (90:170:0.2 and 0.5);
\draw[ line width=0.4mm, rotate=0] (-1.2,2.41) arc (190:270:0.2 and 0.5);
\path[overlay,-, font=\scriptsize,>=latex, line width=0.4mm]
(-5.95,1.44) edge[out=310, in=180] node[] {} (-2.1, 0.7);
\path[overlay,-, font=\scriptsize,>=latex, line width=0.4mm]
(-2.1,0.7) edge[out=0, in=180] node[] {} (-1.2, 2.6);
\path[overlay,-, font=\scriptsize,>=latex, line width=0.4mm]
(-5.95,1.40) edge[out=310, in=180] node[] {} (-2.0, 0.6);
\path[overlay,-, font=\scriptsize,>=latex, line width=0.4mm]
(-2.0,0.6) edge[out=0, in=180] node[] {} (-1.2, 2.4);
\draw[->, >=latex] (-1.173, 2.75) to (-1.163, 2.85);
\end{tikzpicture}
    \caption{$a_i$}
    \label{sequence1}
\end{figure}
\begin{figure}[h!]
    \centering
\begin{tikzpicture}[thick,scale=0.8, every node/.style={scale=0.8}]
\draw [-, rounded corners, thick] (0,0) -- (6,0) -- (6,3) -- (-6,3) -- (-6,0) --(0,0);
\draw (5,1.5) circle (0.5cm);
\draw (-1,1.5) circle (0.5cm);
\draw (1,1.5) circle (0.5cm);
\draw (-5,1.5) circle (0.5cm);
\node at (-5,1.5){\footnotesize$1$};
\node at (-1,1.5){\footnotesize$i$};
\node at (1,1.5){\scriptsize{$g{-}i{+}1$}};
\node at (5,1.5){\footnotesize$g$};
\node at (-3,1.5){$\cdots$};
\node at (3,1.5){$\cdots$};
\node at (-6,1.5){\rule{6mm}{0mm}};
\draw[ line width=0.2mm, rotate=0, dashed] (1,3) arc (90:270:0.2 and 0.5); 
\draw[ line width=0.4mm, rotate=0] (1,3) arc (90:0:0.2 and 0.5);
\draw[ line width=0.4mm, rotate=0] (1,2) arc (-90:-10:0.2 and 0.5);
\path[overlay,-, font=\scriptsize,>=latex, line width=0.4mm]
(1.2,2.4) edge[out=0,in=180] node[] {} (2.1, 0.7); 
\path[overlay,-, font=\scriptsize,>=latex, line width=0.4mm]
(2.1,0.7) edge[out=0,in=230] node[] {} (6, 1.4); 
\path[overlay,-, font=\scriptsize,>=latex, line width=0.4mm]
(1.2,2.5) edge[out=0,in=180] node[] {} (2.2, 0.75); 
\path[overlay,-, font=\scriptsize,>=latex, line width=0.4mm]
(2.2,0.75) edge[out=0,in=230] node[] {} (6, 1.45); 
\draw[->, >=latex] (1.18, 2.7) to (1.17, 2.8);
\filldraw[black] (6,1.5) circle (2pt) node[right] {$\sigma\cdot\star$};
\end{tikzpicture}
    \caption{$\rho_{yz}(a_i)$}
    \label{sequence2}
\end{figure}
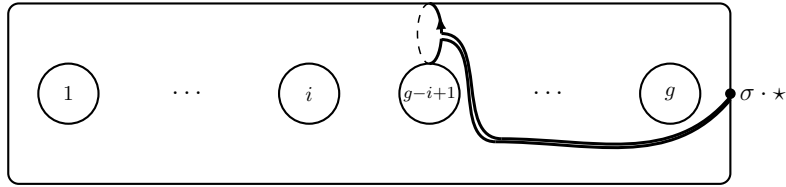
\begin{figure}[h!]
    \centering
\begin{tikzpicture}[thick,scale=0.8, every node/.style={scale=0.8}]
\draw [-, rounded corners, thick] (0,0) -- (6,0) -- (6,3) -- (-6,3) -- (-6,0) --(0,0);
\draw (5,1.5) circle (0.5cm);
\draw (-1,1.5) circle (0.5cm);
\draw (1,1.5) circle (0.5cm);
\draw (-5,1.5) circle (0.5cm);
\node at (-5,1.5){\footnotesize$1$};
\node at (-1,1.5){\footnotesize$i$};
\node at (1,1.5){\scriptsize{$g{-}i{+}1$}};
\node at (5,1.5){\footnotesize$g$};
\node at (-3,1.5){$\cdots$};
\node at (3,1.5){$\cdots$};
\node at (-6,1.5){\rule{6mm}{0mm}};
\draw[ line width=0.2mm, rotate=0, dashed] (1,1) arc (90:270:0.2 and 0.5); 
\draw[ line width=0.4mm, rotate=0] (1,1) arc (90:0:0.2 and 0.5);
\draw[ line width=0.4mm, rotate=0] (1,0) arc (-90:-10:0.2 and 0.5);
\filldraw[black] (6,1.5) circle (2pt) node[right] {$\sigma\cdot \star$};
\path[overlay,-, font=\scriptsize,>=latex, line width=0.4mm]
(2.1, 2.5) edge[out=0, in=109] node[] {} (6, 1.57);
\path[overlay,-, font=\scriptsize,>=latex, line width=0.4mm]
(1.2, 0.5) edge[out=0, in=180] node[] {} (2.1, 2.5);
\path[overlay,-, font=\scriptsize,>=latex, line width=0.4mm]
(2.2,2.44) edge[out=0, in=109] node[] {} (5.95, 1.55);
\path[overlay,-, font=\scriptsize,>=latex, line width=0.4mm]
(1.2, 0.4) edge[out=0, in=180] node[] {} (2.2, 2.44);
\draw[->, >=latex] (0.82, 0.7) to (0.84, 0.8);
\end{tikzpicture}
    \caption{$\rho_{xz}\rho_{yz}(a_i)$}
    \label{sequence3}
\end{figure}
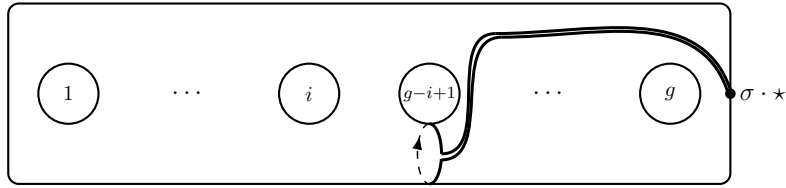
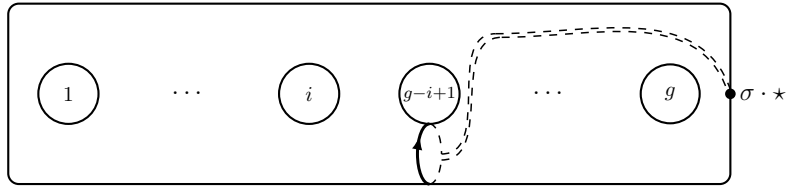
\begin{figure}[h!]
    \centering
\begin{tikzpicture}[thick,scale=0.8, every node/.style={scale=0.8}]
\draw [-, rounded corners, thick] (0,0) -- (6,0) -- (6,3) -- (-6,3) -- (-6,0) --(0,0);
\draw (5,1.5) circle (0.49cm);
\draw (-1,1.5) circle (0.5cm);
\draw (1,1.5) circle (0.5cm);
\draw (-5,1.5) circle (0.5cm);
\node at (-5,1.5){\footnotesize$1$};
\node at (-1,1.5){\footnotesize$i$};
\node at (1,1.5){\scriptsize{$g{-}i{+}1$}};
\node at (5,1.5){\footnotesize$g$};
\node at (-3,1.5){$\cdots$};
\node at (3,1.5){$\cdots$};
\node at (-6,1.5){\rule{6mm}{0mm}};
\draw[ line width=0.4mm, rotate=0] (1,1) arc (90:270:0.2 and 0.5); 
\draw[ line width=0.2mm, rotate=0, dashed] (1,1) arc (90:0:0.2 and 0.5);
\draw[ line width=0.2mm, rotate=0, dashed] (1,0) arc (-90:-10:0.2 and 0.5);
\path[overlay,-, font=\scriptsize,>=latex, line width=0.2mm]
(2.1, 2.5) edge[out=0, in=109, dashed] node[] {} (6, 1.57);
\path[overlay,-, font=\scriptsize,>=latex, line width=0.2mm]
(1.2, 0.5) edge[out=0, in=180, dashed] node[] {} (2.1, 2.5);
\path[overlay,-, font=\scriptsize,>=latex, line width=0.2mm]
(2.2,2.44) edge[out=0, in=109, dashed] node[] {} (5.95, 1.55);
\path[overlay,-, font=\scriptsize,>=latex, line width=0.2mm]
(1.2, 0.4) edge[out=0, in=180, dashed] node[] {} (2.2, 2.44);
\filldraw[black] (6,1.5) circle (2pt) node[right] {$\sigma \cdot \star$};
\draw[->, >=latex] (0.82, 0.7) to (0.84, 0.8);
\end{tikzpicture}
    \caption{$\rho_{xy}\rho_{xz}\rho_{yz}(a_i)$}
    \label{sequence4}
\end{figure}
\begin{figure}[h!]
    \centering
\begin{tikzpicture}[thick,scale=0.8, every node/.style={scale=0.8}]
\draw [-, rounded corners, thick] (0,0) -- (6,0) -- (6,3) -- (-6,3) -- (-6,0) --(0,0);
\draw (5,1.5) circle (0.5cm);
\draw (-1,1.5) circle (0.5cm);
\draw (1,1.5) circle (0.5cm);
\draw (-5,1.5) circle (0.5cm);
\node at (-5,1.5){\footnotesize$1$};
\node at (-1,1.5){\footnotesize$i$};
\node at (1,1.5){\scriptsize{$g{-}i{+}1$}};
\node at (5,1.5){\footnotesize$g$};
\node at (-3,1.5){$\cdots$};
\node at (3,1.5){$\cdots$};
\node at (7,1.5){\rule{0mm}{4mm}};
\filldraw[black] (-6,1.5) circle (2pt) node[left] {$\star$};
\draw[ line width=0.4mm, rotate=0] (1,1) arc (90:270:0.2 and 0.5); 
\draw[ line width=0.2mm, rotate=0, dashed] (1,1) arc (90:0:0.2 and 0.5);
\draw[ line width=0.2mm, rotate=0, dashed] (1,0) arc (-90:-10:0.2 and 0.5);
\path[overlay,-, font=\scriptsize,>=latex, line width=0.4mm]
(-5.95,1.5) edge[out=50,in=180] node[] {} (0, 2.7);
\path[overlay,-, font=\scriptsize,>=latex, line width=0.4mm]
(0,2.7) edge[out=0,in=170] node[] {} (6, 2.4);
\path[overlay,-, font=\scriptsize,>=latex, line width=0.2mm]
(2.2,2.44) edge[out=0, in=190, dashed] node[] {} (6, 2.4);
\path[overlay,-, font=\scriptsize,>=latex, line width=0.2mm]
(1.2, 0.4) edge[out=0, in=180, dashed] node[] {} (2.2, 2.44);
\path[overlay,-, font=\scriptsize,>=latex, line width=0.2mm]
(1.2, 0.5) edge[out=0, in=180, dashed] node[] {} (2.1, 2.5);
\path[overlay,-, font=\scriptsize,>=latex, line width=0.2mm]
(2.1, 2.5) edge[out=0, in=190, dashed] node[] {} (6, 2.5);
\path[overlay,-, font=\scriptsize,>=latex, line width=0.4mm]
(-5.95,1.56) edge[out=50,in=180] node[] {} (0, 2.8);
\path[overlay,-, font=\scriptsize,>=latex, line width=0.4mm]
(0,2.8) edge[out=0,in=165] node[] {} (6, 2.5);
\draw[->, >=latex] (0.82, 0.7) to (0.84, 0.8);
\draw[->, >=latex] (4.3, 2.82) to (4.25, 2.83);
\draw[->, >=latex] (4.25, 2.63) to (4.3, 2.62);
\end{tikzpicture}
    \caption{$\mu_*(a_i)$}
    \label{sequence5}
\end{figure}

\begin{remark}\label{rf}
Regarding the possibility of proving the equality $\TC^{\mathbb{Z}_2}(\Sigma_g)=4$,
it is natural to ask whether the 2-dimensional effective zero-divisor $c$ in Proposition~\ref{ixua} factors as a product of two effective zero-divisors. It can be shown that such a factorization fails even with twisted $\mathbb{Z}$-coefficients.  Yet, it is conceivable that there could exist some other four effective zero-divisors of dimension 1 with non-vanishing product. In this direction we remark that, in genus 2, any two $\widetilde{\mathbb{Z}}$-coefficients effective zero-divisors of dimension 1 have trivial product. Yet, we report that, surprisingly, the classes $\upnu$, $\upnu'$ and $\upnu''$ coming from Proposition~\ref{cdelpdim1} with $g_1=g_2=3$ and, respectively, $S_1=S_2=\{a_2\}$, $S'_1=S'_2=\{b_2\}$ and $S''_1=S''_2=\{a_1,a_3\}$ are effective zero-divisors with non-vanishing product. Such an improvement on product-length of $\mathbb{Z}$-twisted effective zero-divisors would seem to suggest that, in (moderately) larger genera, there could actually exist four $\widetilde{\mathbb{Z}}$-coefficients effective zero-divisors of dimension 1 with non-vanishing product ---which would yield the equality $\TC^{\mathbb{Z}_2}(\Sigma_g)=4$ in the corresponding genera. The exploration of such a possibility demands replacing Lemma~\ref{delosdia} by an explicit description of the map $\mu_*\colon \uppi_g\to\uppi_g$.
\end{remark}


\bigskip

{\sc Departamento de Matem\'aticas

Centro de Investigaci\'on y de Estudios Avanzados del I.P.N.

Av.~Instituto Polit\'ecnico Nacional n\'umero~2508, San Pedro Zacatenco

M\'exico City 07000, M\'exico.}

\tt jesus@math.cinvestav.mx

cadavid@math.cinvestav.mx

\end{document}